\documentclass[a4paper,11pt]{amsart}
\usepackage{amsmath,amsthm,amssymb}
\usepackage[top=30truemm,bottom=30truemm,left=25truemm,right=25truemm]{geometry}
\usepackage{ytableau}
\usepackage{pinlabel}
\usepackage{bm}

\newtheorem{theo}{Theorem}[subsection]
\newtheorem{defi}[theo]{Definition}
\newtheorem{lem}[theo]{Lemma}
\newtheorem{rem}[theo]{Remark}
\newtheorem{prop}[theo]{Proposition}
\newtheorem{cor}[theo]{Corollary}

\newtheorem{ex}[theo]{Example}

\newtheorem{Thm}{Theorem}
  
\newtheorem{Cor}[Thm]{Corollary}

\newenvironment{proof of filtration}{\paragraph{\it{Proof of Theorem \ref{filtration}.}}}{\hfill$\square$}

\newcommand{\nc}{\newcommand}
\nc{\on}{\operatorname}

\nc{\III}{I\hspace{-.1em}I\hspace{-.1em}I }

\nc{\C}{\mathbb{C}}
\nc{\R}{\mathbb{R}}
\nc{\Q}{\mathbb{Q}}
\nc{\Z}{\mathbb{Z}}
\nc{\N}{\mathbb{N}}

\nc{\bbB}{\mathbb{B}}
\nc{\bbI}{\mathbb{I}}
\nc{\bbS}{\mathbb{S}}
\nc{\bbT}{\mathbb{T}}

\nc{\bfa}{\mathbf{a}}
\nc{\bfb}{\mathbf{b}}
\nc{\bfB}{\mathbf{B}}
\nc{\bfC}{\mathbf{C}}
\nc{\bfD}{\mathbf{D}}
\nc{\bfH}{\mathbf{H}}
\nc{\bfi}{\mathbf{i}}
\nc{\bfL}{\mathbf{L}}
\nc{\bfs}{\mathbf{s}}
\nc{\bfS}{\mathbf{S}}
\nc{\bfT}{\mathbf{T}}
\nc{\bfU}{\mathbf{U}}
\nc{\bfV}{\mathbf{V}}
\nc{\V}{\mathbf{V}}
\nc{\bflm}{{\bm \lambda}}
\nc{\bfmu}{{\bm \mu}}
\nc{\bfnu}{{\bm \nu}}

\nc{\clB}{\mathcal{B}}
\nc{\clF}{\mathcal{F}}
\nc{\clFj}{\clF^\jmath}
\nc{\clL}{\mathcal{L}}
\nc{\clO}{\mathcal{O}}
\nc{\clH}{\mathcal{H}}
\nc{\Hc}{\mathcal{H}}

\nc{\Ss}{\mathfrak{S}}
\nc{\frb}{\mathfrak{b}}
\nc{\g}{\mathfrak{g}}
\nc{\frh}{\mathfrak{h}}
\nc{\frl}{\mathfrak{l}}
\nc{\frn}{\mathfrak{n}}
\nc{\frp}{\mathfrak{p}}
\nc{\frgl}{\mathfrak{gl}}
\nc{\frsl}{\mathfrak{sl}}

\nc{\hf}{\frac{1}{2}}
\nc{\inv}{^{-1}}

\nc{\qu}{\quad}
\nc{\qqu}{\qquad}
\nc{\la}{\langle}
\nc{\ra}{\rangle}
\nc{\ol}{\overline}
\nc{\ul}{\underline}
\nc{\vphi}{\varphi}
\nc{\vpi}{\varpi}
\nc{\vep}{\varepsilon}
\nc{\lm}{\lambda}
\nc{\Lm}{\Lambda}
\nc{\Lmj}{\Lm^{\jmath}}

\nc{\Ker}{\on{Ker}}
\nc{\im}{\on{Im}}
\nc{\Hom}{\on{Hom}}
\nc{\End}{\on{End}}
\nc{\Span}{\on{Span}}
\nc{\Lie}{\on{Lie}}
\nc{\Gr}{\on{Gr}}
\nc{\op}{\on{op}}
\nc{\tr}{\on{tr}}
\nc{\id}{\on{id}}
\nc{\rk}{\on{rk}}
\nc{\ad}{\on{ad}}
\nc{\sr}{\on{sr}}
\nc{\ME}{\on{ME}}
\nc{\EM}{\on{EM}}
\nc{\para}{\on{par}}
\nc{\sgn}{\on{sgn}}
\nc{\gr}{\on{gr}}
\nc{\ef}{\on{ef}}
\nc{\up}{\on{up}}
\nc{\low}{\on{low}}

\nc{\IF}{\text{ if }}
\nc{\AND}{\text{ and }}
\nc{\OR}{\text{ or }}
\nc{\OW}{\text{ otherwise}}
\nc{\Forall}{\text{ for all }}
\nc{\Forsome}{\text{ for some }}
\nc{\lowerterms}{\text{(lower terms)}}
\nc{\higherterms}{\text{(higher terms)}}
\nc{\otherterms}{\text{(other terms)}}
\nc{\IFF}{\text{ if and only if }}

\nc{\Lcell}{\underset{L}{\sim}}
\nc{\Rcell}{\underset{R}{\sim}}
\nc{\LRcell}{\underset{LR}{\sim}}

\nc{\Ij}{\bbI^{\jmath}}
\nc{\psij}{\psi^{\jmath}}
\nc{\sigmaj}{\sigma^{\jmath}}
\nc{\ej}{e^{\jmath}}
\nc{\Oj}{\Oint^{\jmath}}
\nc{\Oint}{\mathcal{O}_{\mathrm{int}}}

\nc{\A}{\mathbf{A}}
\nc{\AZ}{\A_{\Z}}
\nc{\Ao}{\A_0}
\nc{\Ai}{\A_{\infty}}
\nc{\Lo}{\clL_0}
\nc{\Li}{\clL_\infty}
\nc{\Gj}{G^{\jmath}}
\nc{\Gjlow}{\Gj_{\mathrm{low}}}
\nc{\Gjup}{\Gj_{\mathrm{up}}}

\nc{\U}{\mathbf{U}}
\nc{\UA}{\U_{\A}}
\nc{\Uj}{\bfU^{\jmath}}
\nc{\Ujo}{\Uj_0}
\nc{\UjA}{\Uj_{\A}}
\nc{\UjLevi}{\Uj_{\on{Levi}}}
\nc{\Ujprime}{\U^{\jmath \prime}}

\nc{\wt}{\on{wt}}
\nc{\wtj}{\wt^{\jmath}}
\nc{\Wtj}{\on{Wt}^{\jmath}}
\nc{\wtjprime}{\wt^{\jmath \prime}}

\nc{\Pj}{P^{\jmath}}
\nc{\Pjprime}{P^{\jmath \prime}}
\nc{\Par}{\on{Par}}
\nc{\DP}{\on{DP}}
\nc{\DPj}{\DP^{\jmath}}
\nc{\SST}{\on{SST}}
\nc{\Cj}{C^{\jmath}}
\nc{\bbBj}{\bbB^{\jmath}}
\nc{\Tj}{T^{\jmath}}

\nc{\pij}{\pi^{\jmath}}
\nc{\IIj}{\ol{\Ij}}
\nc{\Qj}{Q^{\jmath}}
\nc{\etil}{\widetilde{e}}
\nc{\ftil}{\widetilde{f}}
\nc{\xtil}{\widetilde{x}}
\nc{\Etil}{\widetilde{E}}
\nc{\Ftil}{\widetilde{F}}
\nc{\Xtil}{\widetilde{X}}
\nc{\btil}{\widetilde{b}}

\nc{\etilmax}{\etil^{\max}}
\nc{\ftilmax}{\ftil^{\max}}
\nc{\Lambdaj}{\Lambda^{\jmath}}
\nc{\leqj}{\leq^{\jmath}}
\nc{\geqj}{\geq^{\jmath}}
\nc{\simR}{\underset{R}{\sim}}
\nc{\simL}{\underset{L}{\sim}}
\nc{\ch}{\on{ch}}
\nc{\tauj}{\tau^{\jmath}}
\nc{\Psij}{\Psi^{\jmath}}
\nc{\trilefteq}{\trianglelefteq}
\nc{\trileft}{\triangleleft}
\nc{\tririghteq}{\trianglerighteq}
\nc{\triright}{\triangleright}
\nc{\bhat}{\widehat{b}}

\nc{\TBA}{\textcolor{red}{TBA}}

\title[Global crystal bases for modules over a QSP]{Global crystal bases for integrable modules over a quantum symmetric pair of type AIII}
\author[H. Watanabe]{Hideya Watanabe}
\address{(H. Watanabe) Department of Mathematics, Tokyo Institute of Technology, 2-12-1 Oh-okayama, Meguro-ku, Tokyo 152-8550, Japan}
\email{watanabe.h.at@m.titech.ac.jp}
\subjclass[2010]{Primary~17B10}
\keywords{quantum symmetric pair, Hecke algebra, global crystal basis}
\date{\today}
\begin{document}
\maketitle

\begin{abstract}
In this paper, we study basic properties of global $\jmath$-crystal bases for integrable modules over a quantum symmetric pair coideal subalgebra $\Uj$ associated to the Satake diagram of type AIII with even white nodes and no black nodes. Also, we obtain an intrinsic characterization of the $\jmath$-crystal bases, whose original definition is artificial.
\end{abstract}

%\tableofcontents

\section{Introduction}
Let $\U = U_q(\mathfrak{sl}_{2r+1})$ be the quantum group over the field $\Q(q)$ of rational functions in one variable $q$, and $\Uj$ its coideal subalgebra such that $(\U,\Uj)$ forms a quantum symmetric pair of type AIII in the sense of \cite{Le99}. Bao and Wang \cite{BW13} introduced the notion of $\jmath$-canonical bases for the based $\U$-modules. A based $\U$-module is a $\U$-module $M$ with a bar-involution $\psi_M$ and a distinguished basis $\bfB$ satisfying some conditions (see \cite{L94} for the precise definition). One of the key ingredients for the construction of the $\jmath$-canonical bases is the intertwiner (also known as the quasi-$K$-matrix) $\Upsilon$. Using $\Upsilon$, Bao and Wang defined a new involution $\psij_M := \Upsilon \circ \psi_M$ on $M$ that is compatible with the bar-involution $\psij$ on $\Uj$. Then, for each $b \in \bfB$, there exists a unique $b^{\jmath} \in M$ such that $\psij_M(b^{\jmath}) = b^{\jmath}$ and $b^{\jmath} - b \in \bigoplus_{\substack{b' \in \bfB \\ b' \prec^{\jmath} b}} q\Q[q] b'$, where $\preceq^{\jmath}$ is a partial order on $\bfB$. Clearly, $\{ b^{\jmath} \mid b \in \bfB \}$ is a basis of $M$, which is called the $\jmath$-canonical basis of $(M,\bfB)$.

The multi-parameter version of $\Uj$ was considered in \cite{BWW16}. Thanks to the integrality of the intertwiner $\Upsilon$, the notion of $\jmath$-canonical bases can be defined analogously. The condition $b^{\jmath} -  b \in \bigoplus_{\substack{b' \in \bfB \\ b' \prec^{\jmath} b}} q\Q[q] b'$ is replaced by $b^{\jmath} - b \in \bigoplus_{\substack{b' \in \bfB \\ b' \prec^{\jmath} b}} (p \Q[p,q,q\inv] \oplus q\Q[q]) b'$. The general theory of the $\jmath$-canonical bases (usually called $\imath$-canonical bases) for the general quantum symmetric pairs was developed in \cite{BW18}.

In \cite{W17}, the author classified all irreducible $\Uj$-modules in a category $\Oj$, which is an analog of the category $\Oint$ of integrable $\U$-modules, and proved that $\Oj$ is semisimple; the isomorphism classes of irreducible modules in $\Oj$ are classified by the set $\Pj$ of bipartitions of length $(r;r+1)$. When the parameters are in the asymptotic case, to each irreducible module in $\Oj$, the author associated a local basis, $\jmath$-crystal basis, which is an analog of Kashiwara's crystal basis. By the complete reducibility, every object in $\Oj$ admits a $\jmath$-crystal basis. In particular, each $\U$-module in $\Oint$, regarded as a $\Uj$-module, has a $\jmath$-crystal basis.

It should be noted that the $\jmath$-crystal basis of a $\U$-module in $\Oint$ is the localized $\jmath$-canonical basis (\cite[Section 1.3]{W17}). To be precise, let $M \in \Oint$ with a canonical basis (or global crystal basis) $\bfB$. Since $(M,\bfB)$ is a based module, it has a $\jmath$-canonical basis $\{ b^\jmath \mid b \in \bfB \}$. Set $\clL := \Span_{\Ao} \bfB$, where $\Ao := \{ f/g \in \Q(p,q) \mid f,g \in p\Q[p,q,q\inv] \oplus \Q[q],\ \lim_{q \rightarrow 0}(\lim_{p\rightarrow 0} g) \neq 0 \}$. Then, $\clB := \{ b^{\jmath} + q\clL \mid b \in \bfB \}$ is a $\Q$-basis of $\clL/q\clL$, and $(\clL,\clB)$ forms a $\jmath$-crystal basis of $M$. Hence, $b^{\jmath} + q\clL$ can be thought of as the localization of $b^{\jmath}$ at $p = q = 0$. Conversely, we may say that the $\jmath$-canonical basis of a based $\U$-module is a globalization of its $\jmath$-crystal basis.

Here arises a natural question: Does a $\Uj$-module in $\Oj$ that is not a based $\U$-module admit a globalization of its $\jmath$-crystal basis? One of the main result of this paper gives the affirmative answer to this question.

In our strategy, the multi-parameter $q$-Schur duality between $\Uj$ and the Hecke algebra of type $B$ (\cite{BWW16}), and the irreducibility of the Kazhdan-Lusztig cell representations of the asymptotic multi-parameter Hecke algebra of type $B$ (\cite{BI03}) play key roles. Let us recall the latter objects briefly. Kazhdan and Lusztig \cite{KL79} gave a partition $W = \sqcup_{X \in L(W)} X$ of a Coxeter group $W$ into the left cells; here, $L(W)$ denotes the set of left cells. To each left cell $X \in L(W)$, they associated an $\Hc$-module $C^L_X$ which is called the left cell representation corresponding to $X$. The left cell representation $C^L_X$ is defined to be the quotient of a left ideal $C_{\leq_L X}$ of $\Hc$ spanned by some Kazhdan-Lusztig basis elements by its subspace $C_{<_L X}$, which is also spanned by some Kazhdan-Lusztig basis elements. Therefore, $C^L_X$ has a basis consisting of the images of some Kazhdan-Lusztig basis elements under the canonical map $C_{\leq_L X} \twoheadrightarrow C^L_X$. It is known that each left cell representation is irreducible if $W$ is of type $A$. When $W$ is of type $B$, the irreducibility of the left cell representations depend on the choice of the parameters $p,q$. According to \cite{BI03}, the left cell representations are irreducible when the parameters are asymptotic.

By the multi-parameter $q$-Schur duality for type $B$, the tensor power $\V^{\otimes d}$ of the vector representation of $\U$ is equipped with a $(\Uj,\bfH)$-bimodule structure whose irreducible decomposition is multiplicity free, where $\bfH$ denotes the multi-parameter Hecke algebra of type $B$ over the field $\Q(p,q)$ of rational functions in two variables $p,q$. Then, for each $X \in L(W)$, the left $\Uj$-module $\V^{\otimes d} \otimes_{\bfH} \bfC^L_X$ is irreducible, where $\bfC^L_X := \Q(p,q) \otimes_{\Z[p^{\pm1},q^{\pm1}]} C^L_X$. Every irreducible $\Uj$-module can be obtained in this way as $d \geq 1$ varies. The main result of this paper states that the basis of $\V^{\otimes d} \otimes_{\bfH} \bfC^L_X$ induced from the Kazhdan-Lusztig basis of $C^L_X$ is a globalization of the $\jmath$-crystal basis.

Our approach provides the following characterization of the $\jmath$-crystal bases and its globalization of the finite-dimensional irreducible $\Uj$-modules. Let $L \in \Oj$ be irreducible and $v \in L$ a highest weight vector. Define two symmetric bilinear forms $(\cdot, \cdot)_1$ and $(\cdot,\cdot)_2$ on $L$ and an involutive anti-linear automorphism $\psij_L$ on $L$ by
\begin{align}
\begin{split}
&(v,v)_1 = 1,\ (xm,n)_1 = (m,\sigmaj(x)n)_1 \qu \Forall x \in \Uj,\ m,n \in L, \\
&(v,v)_2 = 1,\ (xm,n)_2 = (m,\tauj(x)n)_2 \qu \Forall x \in \Uj,\ m,n \in L, \\
&\psij_L(v) = v,\ \psij_L(xm) = \psij(x) \psij_L(m) \qu \Forall x \in \Uj, m \in L,
\end{split} \nonumber
\end{align}
where $\sigmaj$, $\tauj$, and $\psij$ are automorphisms of $\Uj$ defined in Proposition \ref{automorphisms}.

\begin{Thm}\label{main}
Let $\bflm \in \Pj$, $L(\bflm)$ the corresponding irreducible $\Uj$-module, $(\clL(\bflm),\clB(\bflm))$ the $\jmath$-crystal basis of $L(\bflm)$ such that $v + q\clL(\bflm) \in \clB(\bflm)$. Then, there exist $\Gj(b)$, $b \in \clB(\bflm)$ satisfying the following hold.
\begin{enumerate}
\item $\clL = \{ m \in L \mid (m,m)_2 \in \Ao \}$.
\item $\clB$ forms an orthonormal basis of $\clL/q\clL$ with respect to the symmetric bilinear form induced from $(\cdot,\cdot)_2$.
%\item For each $b \in \clB$, there exists a unique $\Gj(b) \in \clL$ such that $\psij_L(\Gj(b)) = \Gj(b)$, and $(\Gj(b),\Gj(b)) \in 1 + q\Ao$.
\item Set $L_{\A} := \Span_{\A}\{ \Gj(b) \mid b \in \clB \}$, where $\A := \Q[p,p\inv,q,q\inv]$. Then, $(\clL,L_{\A},\psij_L(\clL))$ forms a balanced triple with the global basis $\{ \Gj(b) \mid b \in \clB \}$.
\item $L$ has the basis dual to $\Gj(\clB)$ with respect to $(\cdot,\cdot)_1$.
\end{enumerate}
\end{Thm}

Next, we investigate basic properties of global $\jmath$-crystal basis for not necessarily irreducible $\Uj$-modules. Especially, we roughly describe the matrix coefficients of the actions of the generators of $\Uj$ with respect to a given global $\jmath$-crystal basis.

We end this paper by proving that the global $\jmath$-crystal basis for a $\U$-module (not $\Uj$-module) is compatible with the filtration coming from the dominance order of the bipartitions (see subsection \ref{subsec: global basis} for the definition of this filtration). A similar result is well-known for ordinary global crystal bases (\cite{K93}, \cite{L94}).

\begin{Thm}
Let $M$ be a $\U$-module with a global $\jmath$-crystal basis $\Gj(\clB)$. Then, for each $\bflm \in \Pj$, the subquotient $W_\bflm(M)$ of $M$ has $\{ \Gj(b) + W_{\succ \bflm} \mid I(b) = \bflm  \}$ as a global $\jmath$-crystal basis. Moreover, there exists an isomorphism $L(\bflm)^{\oplus m_\bflm} \rightarrow W_\bflm(M)$ which restricts to a bijection $\{ \Gj(b) \mid b \in \clB(\bflm) \}^{\oplus m_\bflm} \rightarrow \{ \Gj(b) + W_{\succ \bflm} \mid I(b) = \bflm  \}$, where $m_\bflm$ denotes the multiplicity of $L(\bflm)$ in $M$.
\end{Thm}

In particular, if we take $M$ to be an irreducible $\Uj$-module, we obtain the following.

\begin{Cor}
Let $\bflm \in \Pj$. Then, $\{ \Gj(b) \mid b \in \clB(\bflm) \}$ is a unique global $\jmath$-crystal basis for $L(\bflm)$.
\end{Cor}

This paper is organized as follows. In Section \ref{notations}, we prepare necessary notations concerning (bi)partitions and Young (bi)tableaux. In Section \ref{RT for Uj} and \ref{crystal}, we give a brief review of \cite{W17}. In Section \ref{global bases}, we introduce the notion of global $\jmath$-crystal bases, and show that the $\jmath$-canonical bases are examples of them. Sections \ref{KL bases}--\ref{existence of global jcry bases} are devoted to prove the existence theorem for the global $\jmath$-crystal bases of the finite-dimensional irreducible $\Uj$-modules. After studying basic properties of the global $\jmath$-crystal bases in Section \ref{basic properties}, we finally prove the compatibility of the $\jmath$-crystal bases and the filtration associated to the dominance order of the bipartitions in Section \ref{sec:proof of main theorem}.

\subsection*{Acknowledgement}
This work is supported by JSPS KAKENHI grant number 17J00172.

\section{Notations}\label{notations}
Throughout this paper, we fix a positive integer $r$. For $n \in \hf \Z$, set $\ul{n} := n - \hf$. Note that $-\ul{n} = -n + \hf \neq \ul{-n}$. We set
$$
I := \{ -r,\ldots,-1,0,1,\ldots,r \}, \qu \bbI := \{ -\ul{r}, \ldots, -\ul{1}, \ul{1}, \ldots, \ul{r} \}, \qu \Ij := \{ 1,\ldots,r \}.
$$

A partition of $n \in \N$ of length $l \in \N$ is a nonincreasing sequence $\lm = (\lm_1,\ldots,\lm_l)$ of nonnegative integers satisfying $\sum_{i=1}^l \lm_i = n$. Let $|\lm| := n$ and $\ell(\lm) := l$, and call them the size and the length of $\lm$, respectively. We denote by $\Par_l(n)$ the set of partitions of $n$ of length $l$.

We often identify a partition with a Young diagram in a usual way. Let $(L,\preceq)$ be a totally ordered set. A semistandard tableau of shape $\lm \in \Par_l(n)$ in letters $L$ is a filling of the Young diagram $\lm$ with elements of $L$, which weakly increases (with respect to the total order $\preceq$) from left to right along the rows, and strictly increases from the top to the bottom along the columns.

A bipartition of $n \in \N$ of length $(l;m) \in \N^2$ is an ordered pair $\bflm := (\bflm^-;\bflm^+)$ of partitions such that $\ell(\bflm^-) = l$, $\ell(\bflm^+) = m$, and $|\bflm| := |\bflm^-|+|\bflm^+| = n$. We denote by $P_{(l;m)}(n)$ the set of bipartitions of $n$ of length $(l;m)$. For totally ordered sets $(L^-,\preceq^-)$ and $(L^+,\preceq^+)$, a semistandard tableau of shape $\bflm \in P_{(l;m)}(n)$ in letters $(L^-;L^+)$ is an ordered pair $\bfT = (T^-;T^+)$, where $T^\pm$ is a semistandard tableau of shape $\bflm^\pm$ in letters $L^\pm$.

For partitions $\mu \subset \lm$, define the skew partition $\lm/\mu$ in a usual way. For bipartitions $\bfmu \subset \bflm$ (i.e., $\bfmu^- \subset \bflm^-$ and $\bfmu^+ \subset \bflm^+$), define the skew bipartition $\bflm/\bfmu$ to be $(\bflm^-/\bfmu^-;\bflm^+/\bfmu^+)$. A skew partition $\lm/\mu$ is said to be a horizontal strip if each column of $\lm/\mu$ contains at most one box. We say that a skew bipartition $\bflm/\bfmu$ is a horizontal strip if $\bflm^\pm/\bfmu^\pm$ are.

Set
\begin{itemize}
\item $P(n) = P_r(n) := \Par_{2r+1}(n)$: the set of partitions of $n$ of length $2r+1$.
\item $P = P_r := \bigsqcup_{n \in \N} P(n)$: the set of partitions of length $2r+1$.
\item $\Par_l := \bigsqcup_{n \in \N} \Par_l(n)$: the set of partitions of length $l$.
\item $\Pj(n) = \Pj_r(n) := P_{(r+1;r)}(n)$: the set of bipartitions of $n$ of length $(r+1;r)$.
\item $\Pj = \Pj_r := \bigsqcup_{n \in \N} \Pj(n)$: the set of bipartitions of length $(r+1;r)$.
\item $\SST(\lm)$: the set of semistandard tableaux of shape $\lm \in P(n)$ in letters $I$.
\item $\SST(\bflm)$: the set of semistandard tableaux of shape $\bflm \in \Pj(n)$ in letters $(I \setminus \Ij;\Ij)$ with total orders $0 \prec^- -1 \prec^- \cdots \prec^- -r$ and $1 \prec^+ \cdots \prec^+ r$.
\end{itemize}

For $\bflm \in \Pj$, we refer the $i$-th row of $\bflm^-$ to as the $-(i-1)$-th row of $\bflm$, and the $j$-th row of $\bflm^+$ to as the $j$-th row of $\bflm$. Also, for $i \in I$, set $\bflm_i$ to be the length of the $i$-th row of $\bflm$, i.e.,
$$
\bflm_i := \begin{cases}
\bflm^-_{-i+1} \qu & \IF i \leq 0, \\
\bflm^+_i \qu & \IF i > 0.
\end{cases}
$$

For $i \in \Ij$, set $\bflm \downarrow_i := (\bflm_0,\bflm_{-1},\ldots,\bflm_{-i};\bflm_1,\ldots,\bflm_i) \in \Pj_i$.

For $\bfT \in \SST(\bflm)$ and $i \in \Ij$, set $\bfT \downarrow_i$ to be the semistandard tableau obtained from $\bfT$ by deleting the boxes whose entries are less than $-i$ or greater than $i$.

For each $\bflm \in \Pj$, let $\bfT_\bflm \in \SST(\bflm)$ be the unique semistandard tableau of shape $\bflm$ whose entries in the $i$-th row are $i$. Note that we have $\bfT_\bflm \downarrow_i = \bfT_{\bflm \downarrow_i}$. For $\bfT \in \SST(\bflm)$ and $i \in I$, set $\bfT(i)$ to be the number of boxes of $\bfT$ whose entries are $i$.

\begin{defi}\normalfont
\ 
\begin{enumerate}
\item $\preceq$ is a partial order (called the dominance order) on $\Par_l$ defined as follows. For $\lm, \mu \in \Par_l$, we have $\lm \preceq \mu$ if
\begin{enumerate}
\item $|\lm| = |\mu|$ and
\item $\sum_{i=1}^j \lm_i \leq \sum_{i=1}^j \mu_i$ for all $1 \leq j \leq l$.
\end{enumerate}
\item $\preceq$ is a partial order (also called the dominance order) on $\Pj$ defined as follows. For $\bflm, \bfmu \in \Pj$, we have $\bflm \preceq \bfmu$ if
\begin{enumerate}
\item $|\bflm| = |\bfmu|$,
\item $\sum_{i=0}^j \bflm_{-i} \leq \sum_{i=0}^j \bfmu_{-i}$ for all $0 \leq j \leq r$, and
\item $|\bflm^-| + \sum_{i=1}^j \bflm_i \leq |\bfmu^-| + \sum_{i=1}^j \bfmu_i$ for all $1 \leq j \leq r$.
\end{enumerate}
\item $\trilefteq$ is a partial order on $\Pj$ defined as follows. For $\bflm,\bfmu \in \Pj$, we have $\bflm \trilefteq \bfmu$ if $\bflm^- \preceq \bfmu^-$ (dominance order on $\Par_{r+1}$) and $\bflm^+ \preceq \bfmu^+$ (dominance order on $\Par_r$).
\end{enumerate}
\end{defi}

Clearly, $\bflm \trilefteq \bfmu$ implies $\bflm \preceq \bfmu$.

\section{Representation theory of $\Uj$}\label{RT for Uj}
Let $p$ and $q$ be independent indeterminates.
\subsection{Definition of $\Uj$}
Let $\Lm$ be the free $\Z$-module with a free basis $\{ \epsilon_i \mid i \in I \}$, and with a symmetric bilinear form $(\cdot,\cdot)$ defined by $(\epsilon_i,\epsilon_j) = \delta_{i,j}$. For $i \in \bbI$, set
$$
\alpha_i := \epsilon_{\ul{i}}-\epsilon_{\ul{i+1}}, \qu \qu Q := \sum_{i \in \bbI} \Z \alpha_i, \qu \qu Q_+ := \sum_{i \in \bbI} \Z_{\geq 0} \alpha_i.
$$
For $\lm,\mu \in \Lm$, we write $\mu \leq \lm$ if $\lm - \mu \in Q_+$. This defines a partial order on $\Lm$.

The quantum group $\bfU = \bfU_{2r+1} = U_q(\mathfrak{sl}_{2r+1})$ of type $A_{2r}$ is an associative algebra over $\Q(p,q)$ with generators $E_i,F_i,K_i^{\pm 1}$, $i \in \bbI$ subject to the following relations: For $i,j \in \bbI$,
\begin{align}
\begin{split}
&K_i K_i\inv = K_i\inv K_i = 1, \\
&K_i K_j = K_j K_i, \\
&K_i E_j K_i\inv = q^{(\alpha_i,\alpha_j)} E_i, \\
&K_i F_j K_i\inv = q^{-(\alpha_i,\alpha_j)} F_i, \\
&E_i F_j - F_j E_i = \delta_{i,j} \frac{K_i - K_i\inv}{q - q\inv}, \\
&E_i^2 E_j - (q+q\inv)E_iE_jE_i + E_jE_i^2 = 0 \qu \IF |i-j| = 1, \\
&F_i^2 F_j - (q+q\inv)F_iF_jF_i + F_jF_i^2 = 0 \qu \IF |i-j| = 1, \\
&E_i E_j - E_j E_i = 0 \qu \IF |i-j| > 1, \\
&F_i F_j - F_j F_i = 0 \qu \IF|i-j| > 1.
\end{split} \nonumber
\end{align}

In this paper, we use the comultiplication $\Delta$ of $\bfU$ given by
\begin{align}
\begin{split}
&\Delta(K_i^{\pm1}) = K_i^{\pm1} \otimes K_i^{\pm1}, \qu \Delta(E_i) = 1 \otimes E_i + E_i \otimes K_i\inv, \\
&\Delta(F_i) = F_i \otimes 1 + K_i \otimes F_i \qu i \in \bbI.
\end{split} \nonumber
\end{align}

Let $(\bfU,\Uj)$ denote the quantum symmetric pair over $\Q(p,q)$ of type AIII, that is, $\Uj = \Uj_r$ is the subalgebra of $\bfU$ generated by
\begin{align}
&k_i^{\pm1} := (K_{\ul{i}}K_{-\ul{i}})^{\pm1}, \nonumber\\
&e_i := E_{\ul{i}} + p^{-\delta_{i,1}} F_{-\ul{i}} K_{\ul{i}}\inv, \nonumber\\
&f_i := E_{-\ul{i}} + p^{\delta_{i,1}} K_{-\ul{i}}\inv F_{\ul{i}}, \qu i \in \Ij. \nonumber
\end{align}

The $\Uj$ has the following defining relations (\cite{Le99}, see also \cite{BW13}, \cite{BWW16}): For $i,j \in \Ij$,
\begin{align}
\begin{split}
&k_i k_i\inv = k_i\inv k_i = 1, \\
&k_i k_j = k_j k_i, \\
&k_i e_j k_i\inv = q^{(\alpha_{\ul{i}}-\alpha_{-\ul{i}},\alpha_{\ul{j}})} e_j, \\
&k_i f_j k_i\inv = q^{-(\alpha_{\ul{i}}-\alpha_{-\ul{i}},\alpha_{\ul{j}})} f_j, \\
&e_i f_j - f_j e_i = \delta_{i,j} \frac{k_i - k_i\inv}{q - q\inv} \qu \IF (i,j) \neq (1,1), \\
&e_i^2 e_j - (q+q\inv)e_ie_je_i + e_je_i^2 = 0 \qu \IF |i-j| = 1, \\
&f_i^2 f_j - (q+q\inv)f_if_jf_i + f_jf_i^2 = 0 \qu \IF |i-j| = 1, \\
&e_i e_j - e_j e_i = 0 \qu \IF |i-j| > 1, \\
&f_i f_j - f_j f_i = 0 \qu \IF |i-j| > 1, \\
&e_1^2 f_1 - (q+q\inv)e_1f_1e_1 + f_1e_1^2 = -(q+q\inv) e_1 (pqk_1 + p\inv q\inv k_1\inv), \\
&f_1^2 e_1 - (q+q\inv)f_1e_1f_1 + e_1f_1^2 = -(q+q\inv) (pqk_1 + p\inv q\inv k_1\inv) f_1.
\end{split} \nonumber
\end{align}

\begin{prop}\label{automorphisms}
\begin{enumerate}
\item {\cite[Lemma 6.1 (3)]{BW13}} There exists a unique $\Q$-algebra automorphism $\psij$ of $\Uj$ which maps $e_i,f_i,k_i,p,q$ to $e_i,f_i,k_i\inv,p\inv,q\inv$, respectively.
\item There exists a unique $\Q(p,q)$-algebra anti-automorphism $\sigmaj$ of $\Uj$ which maps $e_i,f_i,k_i$ to $f_i,e_i,k_i$, respectively.
\item {\cite[Proposition 4.6]{BW18}} There exists a unique $\Q(p,q)$-algebra anti-automorphism $\tauj$ of $\Uj$ which maps $e_i,f_i,k_i$ to $p^{-\delta_{i,1}}q\inv k_i\inv f_i$, $p^{\delta_{i,1}}q e_i k_i$, $k_i$, respectively.
\end{enumerate}
\end{prop}

\begin{proof}
It suffices to show that the images of the generators of $\Uj$ satisfy the defining relations of $\Uj$; it is straightforward.
\end{proof}

\begin{rem}\normalfont
We have similar automorphisms on $\U$:
\begin{enumerate}
\item There exists a unique $\Q$-algebra automorphism $\psi$ of $\U$ which maps $E_i,F_i,K_i,p,q$ to $E_i,F_i,K_i\inv,p\inv,q\inv$, respectively.
\item There exists a unique $\Q(p,q)$-algebra anti-automorphism $\sigma$ of $\Uj$ which maps $E_i,F_i,K_i$ to $F_i,E_i,K_i$, respectively.
\item There exists a unique $\Q(p,q)$-algebra anti-automorphism $\tau$ of $\Uj$ which maps $E_i,F_i,K_i$ to $q F_i K_i\inv$, $q\inv K_i E_i$, $K_i$, respectively.
\end{enumerate}
Note that $\tauj$ is the restriction of $\tau$ {\cite[Proposition 4.6]{BW18}}, while the others are not.
\end{rem}

Let $\bfU(\frl)$ denote the subalgebra of $\U$ generated by $E_i,F_i,K_j^{\pm1}$, $i \in \bbI \setminus \{ \ul{1} \}$, $j \in \bbI$. Note that we have $e_i,f_i,k_j \in \bfU(\frl)$ for all $i \in \Ij \setminus \{1\}$, $j \in \Ij$. Note that $\bfU(\frl)$ is the quantum group of type $A_r \times A_{r-1}$ with weight lattice $\Lm$.

\subsection{Category $\Oj$}
Let us extend the bilinear form $(\cdot,\cdot)$ on $\Lm$ to $\Lm_\R := \R \otimes_\Z \Lm$. Set $\beta_i := \alpha_{\ul{i}}-\alpha_{-\ul{i}}$, $i \in \Ij$, and $J := \{ \lm \in \Lm_\R \mid (\beta_i,\lm) = 0 \Forall i \in \Ij \}$. Then, the induced bilinear form $(\cdot, \cdot) :(\sum_{i \in \Ij} \R \beta_i) \times (\Lm_\R/J) \rightarrow \R$ denoted by the same symbol is nondegenerate. Let $\delta_j \in \Lm_\R/J$ be such that $(\beta_i,\delta_j) = \delta_{i,j}$ for all $i,j \in \Ij$. Set $\Lmj := \sum_{i \in \Ij} \Z \delta_i$. Let $\gamma_i := \alpha_{\ul{i}} + J \in \Lmj$, $i \in \Ij$, and $\Qj_+ := \sum_{i \in \Ij} \Z_{\geq 0} \gamma_i \subset \Lmj$. For $\lambda,\mu \in \Lambdaj$, we write $\mu \leqj \lambda$ if $\lambda - \mu \in \Qj_+$. This defines a partial order on $\Lmj$.

For a $\Uj$-module $M$ and $\lambda \in \Lambdaj$, we call $M_\lambda := \{ m \in M \mid k_i m = q^{(\beta_i,\lambda)} m \Forall i \in \Ij \}$ the weight space of $M$ of weight $\lambda$. The category $\Oj$ is the full subcategory of the category of all $\Uj$-modules consisting of $\Uj$-modules $M$ satisfying the following:
\begin{itemize}
\item $M$ has a weight space decomposition, i.e., $M = \bigoplus_{\lambda \in \Lambdaj} M_\lambda$.
\item Each weight space of $M$ is finite-dimensional.
\item There exist $\mu_1,\ldots,\mu_l \in \Lmj$ such that if $M_\lm \neq 0$, then $\lm \leqj \mu_i$ for some $i = 1,\ldots,l$.
\item The $f_i$'s act on $M$ locally nilpotently.
\end{itemize}

\begin{theo}[\cite{W17}]
The following hold:
\begin{enumerate}
\item {\cite[Theorem 4.4.3]{W17}} $\Oj$ is semisimple.
\item {\cite[Corollary 7.6.3, 7.6.4]{W17}} Each irreducible $\Uj$-module in $\Oj$ is isomorphic to the irreducible highest weight module $L(\bflm)$ with highest weight $\bflm$ (in the sense of \cite{W17}) for some $\bflm \in \Pj$.
\item For $\bflm,\bfmu \in \Pj$, we have $L(\bflm) \simeq L(\bfmu)$ if and only if $\bflm_i - \bfmu_i$ is constant as $i$ runs through $-r,\ldots,r$.
\end{enumerate}
\end{theo}

\begin{rem}\normalfont
The last statement follows from the definition of $L(\bflm)$.
\end{rem}

For each $\bflm \in \Pj$, let $\wtj(\bflm) \in \Lmj$ denote the weight of a highest weight vector of $L(\bflm)$, namely,
$$
\wtj(\bflm) := \sum_{i \in \Ij} (\bflm_{i-1}-\bflm_i + \bflm_{-(i-1)}-\bflm_{-i}) \delta_i.
$$

\section{Crystal basis theory}\label{crystal}
\subsection{Crystal bases}\label{sec:crystal bases}
The notion of crystal bases (or local bases at $q = 0$) for integrable modules over quantum groups was introduced independently by Kashiwara and Lusztig in different ways (\cite{K90}, \cite{L90a}). Although we will not review the detail, we formulate here some notations concerning the crystal bases. Let $\Oint$ denote the full subcategory of the BGG-category $\clO$ for $\bfU$ consisting of the integrable modules. Let $\Etil_i, \Ftil_i$, $i \in \bbI$ denote the Kashiwara operators. Let $M \in \Oint$, $(\clL,\clB)$ be its crystal basis. For $b \in \clB$ and $i \in \bbI$, set
$$
\vep_i(b) := \max\{ n \mid \Etil_i^nb \neq 0 \}, \qu \vphi_i(b) := \max\{ n \mid \Ftil_i^nb \neq 0 \}.
$$
Also, $\wt(b) \in \Lm$ denotes the weight of $b$.

Recall that, for each $\lm \in P$, the irreducible module $L(\lm)$ has a unique crystal basis $(\clL(\lm),\clB(\lm))$, which is identical to $\SST(\lm)$. For each $M \in \Oint$ with a crystal basis $(\clL,\clB)$, we have a unique irreducible decomposition $\clB = \bigsqcup_{i=1}^l \clB_i$, where $\clB_i \simeq \clB(\lm_i)$ for some $\lm_i \in P$. By retaking $\lm_i$'s if necessary, we may assume that $|\lm_i| - |\lm_j| < 2r+1$ for all $i,j \in \{1,\ldots,l\}$, and that there exists $i$ such that $(\lm_i)_{2r+1} = 0$. Then, $\lm_i$'s are uniquely determined; we set $P(M) = P_r(M) := \{ \lm_1,\ldots,\lm_l \}$. For $b \in \clB$, we define $I(b) = I_r(b) \in P(M)$ to be $\lm_i$ if $b \in \clB_i$. Also let $C(b) = C_r(b) \subset \clB$ denote the connected component of $\clB$ containing $b$. Furthermore, if we write $b = \Ftil_{i_1} \cdots \Ftil_{i_l}b_0$ for some $i_1,\ldots,i_l \in \bbI$, where $b_0$ denotes the highest weight vector in $C(b)$, then define $T_b \in \SST(I(b))$ by $T_b := \Ftil_{i_1} \cdots \Ftil_{i_l} T_0$, where $T_0 \in \SST(I(b))$ corresponding to $b_0 \in C(b) = \clB(I(b))$.

\subsection{$\jmath$-crystal bases}
In \cite{W17}, the notion of $\jmath$-crystal bases was introduced. Let us recall some properties briefly.

Set $\A := \Q[p,p\inv,q,q\inv]$. We denote by $\A_0$ the subring of $\Q(p,q)$ consisting of all elements of the form $f/g$ with $f,g \in p\Q[p,q,q\inv] \oplus \Q[q]$, $\lim_{q \rightarrow 0} (\lim_{p \rightarrow 0} g) \neq 0$. Let $\IIj := \Ij \sqcup \{ 2',\ldots,r' \}$. The Kashiwara operators are denoted by $\etil_i$ and $\ftil_i$, $i \in \IIj$.

The following are basic results for the crystal basis theory of $\Uj$.

\begin{theo}[{\cite[Theorem 7.7.3]{W17}}]
Let $\bflm \in \Pj$, $v_\bflm \in L(\bflm)$ be a highest weight vector. Set
\begin{align}
\begin{split}
&\clL(\bflm) := \Span_{\Ao} \{ \ftil_{i_1} \cdots \ftil_{i_l} v_\bflm \mid l \in \Z_{\geq 0},\ i_1,\ldots,i_l \in \IIj \}, \\
&\clB(\bflm) := \{ \ftil_{i_1} \cdots \ftil_{i_l} v_\bflm + q\clL(\bflm) \mid l \in \Z_{\geq 0},\ i_1,\ldots,i_l \in \IIj \} \setminus \{ 0 \}.
\end{split} \nonumber
\end{align}
Then, $(\clL(\bflm),\clB(\bflm))$ is a unique $\jmath$-crystal basis of $L(\bflm)$. Moreover, $\clB(\bflm)$ is identical to $\SST(\bflm)$; $v_\bflm + q\clL(\bflm) \in \clB(\bflm)$ corresponds to $\bfT_\bflm \in \SST(\bflm)$.
\end{theo}

\begin{theo}\label{restriction of crystal}
Suppose that $M \in \Oint$ has a crystal basis $(\clL,\clB)$. Then, as a $\Uj$-module, $M$ has a $\jmath$-crystal basis whose underlying sets are equal to $(\clL,\clB)$.
\end{theo}

\begin{proof}
This is an easy consequence of \cite[Corollary 7.7.4]{W17}.
\end{proof}

Let $M \in \Oj$ with a $\jmath$-crystal basis $(\clL,\clB)$. For each $b \in \clB$ and $i \in \IIj$, define $\vep_i(b)$, $\vphi_i(b)$, $\wtj(b) \in \Lmj$, $\Pj(M) = \Pj_r(M) \subset \Pj$, $I^\jmath(b) = I^\jmath_r(b) \in \Pj(M)$, $\Cj(b) = \Cj_r(b) \subset \clB$, and $T^\jmath_b \in \SST(I^\jmath(b))$ in a similar way to Section \ref{sec:crystal bases}.

%\begin{defi}\normalfont
%Let $M \in \Oj$ with a $\jmath$-crystal basis $(\clL,\clB)$. We say that $b,b' \in \clB$ are strongly connected if there exists a sequence $\xtil_{i_1}, \ldots, \xtil_{i_l}$ of $\etil_i,\ftil_i$, $i \in \Ij$ such that $b' = \xtil_{i_1} \cdots \xtil_{i_l} b$.
%\end{defi}
%
%
%\begin{defi}\normalfont
%Let $M \in \Oj$ with a $\jmath$-crystal basis $(\clL,\clB)$. For each $b \in \clB$, there exists a unique $b_0 \in \clB$ such that $\etil_i b_0 = 0$ for all $i \in \IIj$ and there exists a sequence $i_1,\ldots,i_l \in \IIj$ such that $b = \ftil_{i_1} \cdots \ftil_{i_l} b_0$. Then, we set
%$$
%\ell(b) := \min\{ \text{the number of the elements of $\IIj$ in $i_1,\ldots,i_l$} \mid b = \ftil_{i_1} \cdots \ftil_{i_l} b_0 \},
%$$
%and call it the level of $b$.
%\end{defi}

\section{Global bases}\label{global bases}
\subsection{Balanced triples}
Let $\ol{\ \cdot \ }$ be the $\Q$-linear automorphism of $\Q(p,q)$ sending $p$ and $q$ to $p\inv$ and $q\inv$, respectively. Set $\Ai := \ol{\A_0}$.

\begin{defi}\normalfont
Let $V$ be a $\Q(p,q)$-vector space and $x \in \{ 0,\emptyset,\infty \}$. An $\A_x$-lattice of $V$ is a free $\A_x$-submodule $U_x$ of $V$ of rank $\dim_{\Q(p,q)} V$ such that $\Q(p,q) \otimes_{\A_x} U_x = V$.
\end{defi}

\begin{defi}[{\cite[Definition 2.1.2]{K93}}]\normalfont
Let $V$ be a $\Q(p,q)$-vector space, $U_x$ an $\A_x$-lattice of $V$ for $x \in \{ 0,\emptyset,\infty \}$. The triple $(U_0,U,U_\infty)$ is said to be balanced if the canonical map
\begin{align}
U_0 \cap U \cap U_\infty \rightarrow U_0/qU_0 \nonumber
\end{align}
is an isomorhism of $\Q$-vector spaces.
\end{defi}

Let $V$ be a $\Q(p,q)$-vector space with a balanced triple $(U_0,U,U_\infty)$. Take a $\Q$-basis $\clB$ of $U_0/qU_0$. Since we have an isomorphism $G:U_0/qU_0 \rightarrow U_0 \cap U \cap U_\infty$ of $\Q$-vector spaces, which is the inverse of the canonical map $U_0 \cap U \cap U_\infty \rightarrow U_0 / qU_0$, we obtain an $\A_x$-basis $G(\clB) = \{ G(b) \mid b \in \clB \}$ of $U_x$ for each $x \in \{ 0,\emptyset,\infty \}$. We call $G(\clB)$ the global basis of $V$ associated to the balanced triple $(U_0,U,U_\infty)$ and the basis $\clB$.

\begin{lem}\label{sub and quotient of bt}
Let $V,U_0,U,U_\infty,\clB,G$ be as above. Take a subset $\clB' \subset \clB$ and set $U'_x$ to be the $\A_x$-span of $G(\clB') := \{ G(b) \mid b \in \clB' \}$ for each $x \in \{ 0,\emptyset,\infty \}$. Also, let $V'$ be the $\Q(p,q)$-span of $G(\clB')$. Then, the following hold:
\begin{enumerate}
\item $(U'_0,U',U'_\infty)$ is a balanced triple with the global basis $G(\clB')$.
\item $(U_0/U'_0,U/U',U_\infty/U'_\infty)$ is a balanced triple with the global basis $\{ G(b) + V' \mid b \in \clB \setminus \clB' \}$.
\end{enumerate}
\end{lem}

\subsection{Global crystal bases and global $\jmath$-crystal bases}
Let $\UA$ denote the $\A$-subalgebra of $\U$ generated by $E_i^{(n)}, F_i^{(n)}, K_i^{\pm1}$, $i \in \bbI$, $n \in \Z_{> 0}$. Similaly, define $\UjA$ to be the $\A$-subalgebra of $\Uj$ generated by $e_i^{(n)}, f_i^{(n)}, k_i^{\pm1}$, $i \in \Ij$, $n \in \Z_{> 0}$.

\begin{lem}[{\cite[1.3.5]{L94}}]\label{q-binomial formula}
Let $A$ be a $\Q(q)$-algebra, $x,y \in A$ such that $xy = q^2yx$. Then, for each $n \in \Z_{> 0}$, we have
$$
(x+y)^n = \sum_{t=0}^n q^{t(n-t)} {n \brack t} y^t x^{n-t}.
$$
\end{lem}

\begin{lem}\label{Uja in UA}
We have $\UjA \subset \UA$.
\end{lem}

\begin{proof}
It suffices to show that $e_i^{(n)},f_i^{(n)} \in \UA$ for all $i \in \Ij$, $n \in \Z_{>0}$. We prove $e_i^{(n)} \in \UA$; the proof for $f_i^{(n)} \in \UA$ is similar. Setting $x := E_{\ul{i}}$ and $y := p^{-\delta_{i,1}}F_{-\ul{i}} K_{\ul{i}}\inv$, we see that
$$
e_i = x+y, \qu xy = q^2yx.
$$
Then, we can apply Lemma \ref{q-binomial formula}, and obtain
$$
e_i^{(n)} = \sum_{t=0}^n q^{t(n-t)} y^{(t)} x^{(n-t)}.
$$
It is easy to see that $y^{(t)} = p^{-\delta_{i,1}t} q^{-\delta_{i,1} \frac{t(t-1)}{2}} F_{-\ul{i}}^{(t)} K_{\ul{i}}^t \in \UA$. Hence, the assertion follows.
\end{proof}

Let $V$ be a $\U$-module in $\Oint$ (resp., $\Uj$-module in $\Oj$) with a crystal basis $(\clL,\clB)$ (resp., $\jmath$-crystal basis $(\clL,\clB)$). Assume that $V$ admits a $\Q$-linear involution $\ol{\ \cdot \ }$ satisfying the following:
\begin{align}
\begin{split}
&\ol{xv} = \psi(x) \ol{v}, \qu \Forall x \in \U,\ v \in V \\
(\text{resp., }&\ol{xv} = \psij(x) \ol{v}, \qu \Forall x \in \Uj,\ v \in V).
\end{split} \nonumber
\end{align}
We call such an involution a $\psi$-involution (resp., $\psij$-involution) on $V$. Since $\clL$ is an $\Ao$-lattice of $V$, $\ol{\clL}$ is an $\Ai$-lattice of $V$.

\begin{defi}\normalfont
Let $V,\clL,\clB,\ol{\ \cdot \ }$ be as above. $V$ is said to have a global crystal basis (resp., global $\jmath$-crystal basis) if there exists a $\UA$-submodule (resp., $\UjA$-submodule) $V_\A$ of $V$ which is an $\A$-lattice forming a balanced triple $(\clL,V_\A,\ol{\clL})$. The associated global basis $G(\clB)$ (resp., $\Gj(\clB)$) is called a global crystal basis (resp., global $\jmath$-crystal basis) of $V$.
\end{defi}

\begin{ex}\normalfont
Let $\bflm \in \Pj_1$ and consider the irreducible $\Uj_1$-module $L(\bflm)$. Recall that $L(\bflm)$ is $(\bflm_0-\bflm_{-1}+1)$-dimensional with a basis $\Gj(\bflm) := \{ f_1^{(n)}v \mid 0 \leq n \leq \bflm_0-\bflm_{-1} \}$, where $v$ denotes a highest weight vector. Also, $L(\bflm)$ has a $\jmath$-crystal basis $(\clL(\bflm),\clB(\bflm))$, where $\clL(\bflm)$ is the $\Ao$-span of $\Gj(\bflm)$, and $\clB(\bflm) = \{ f_1^{(n)}v + q\clL(\bflm) \mid 0 \leq n \leq \bflm_0-\bflm_{-1} \}$. Set $L(\bflm)_\A$ to be the $\A$-span of $\Gj(\bflm)$. Note that there exists a unique $\psij$-involution $\psij_{\bflm}$ on $L(\bflm)$ fixing $v$. Then, $(\clL(\bflm), L(\bflm)_\A, \psij_\bflm(\clL(\bflm)))$ is a balanced triple, and $\Gj(\bflm)$ is a global $\jmath$-crystal basis of $L(\bflm)$.
\end{ex}

\begin{prop}
Let $M \in \Oj$ with a global crystal $\jmath$-crystal basis $\Gj(\clB_M)$, and $N \in \Oint$ with a global crystal basis $G(\clB_N)$ Then, $M \otimes N$ has a global $\jmath$-crystal basis of the form
\begin{align}
\begin{split}
&\{ \Gj(b_1) \diamondsuit G(b_2) \mid b_1 \in \clB_M,\ b_2 \in \clB_N \}. \\
&\Gj(b_1) \diamondsuit G(b_2) \in \Gj(b_1) \otimes G(b_2) + \sum_{\substack{b'_1 \in \clB_M,\ b'_2 \in \clB_N \\ \wt(b'_2) < \wt(b_2)}} a_{b'1,b'_2;b_1,b_2} \Gj(b'_1) \otimes G(b'_2), \qu a_{b'1,b'_2;b_1,b_2} \in \A.
\end{split} \nonumber
\end{align}
\end{prop}

\begin{proof}
The fact that $\clB_M \otimes \clB_N$ forms a $\jmath$-crystal basis of $M \otimes N$ is proved in \cite{W17}. Now, one can construct a global $\jmath$-crystal basis of $M \otimes N$ with the desired property in the same way as the proof of \cite[Theorem 4]{BWW18}.
\end{proof}

\subsection{$\jmath$-canonical bases}
In this subsection, we recall the notion of $\jmath$-canonical bases, which was introduced by H. Bao and W. Wang in \cite{BW13}, and explain that $\jmath$-canonical bases are global $\jmath$-crystal bases. One of the key ingredients for a construction of $\jmath$-canonical bases is the intertwiner $\Upsilon$:

\begin{defi}[{\cite[Theorem 6.4]{BW13}}]
Let $\U^-$ denote the subalgebra of $\U$ generated by $F_i$, $i \in \bbI$. For each $\lm \in Q_+$, there exists a unique $\Upsilon_\lm \in \U^-_{-\lm}$ satisfying the following:
\begin{itemize}
\item $\Upsilon_0 = 1$, \\
\item $\Upsilon := \sum_{\lm \in Q_+} \Upsilon_\lm$ satisfies $\psij(x)\Upsilon = \Upsilon \psi(x)$ for all $x \in \Uj$.
\end{itemize}
\end{defi}

\begin{lem}[{\cite[Proposition 6.12]{BW13}}]
Let $M \in \Oint$ with a $\psi$-involution $\psi_M$. Then, the composite $\Upsilon \circ \psi_M$ is a $\psij$-involution of $M$.
\end{lem}

\begin{theo}[{\cite[Theorem 6.24]{BW13}}]\label{j-canonical}
Let $M \in \Oint$ have a global crystal basis $G(\clB)$ with a crystal basis $(\clL,\clB)$, a $\psi$-involution $\psi_M$, and an $\A$-lattice $M_\A$. Set $\psij_M := \Upsilon \circ \psi_M$. Then, for each $b \in \clB$, there exists a unique $\Gj(b) \in M$ satisfying the following.
\begin{enumerate}
\item $\psij_M(\Gj(b)) = \Gj(b)$.
\item $\Gj(b) = G(b) + \sum_{b' \in B} c_{b',b} G(b')$ for some $c_{b',b} \in q\Ao \cap \A$. Moreover, $c_{b',b} = 0$ unless $\wtj(b') = \wtj(b)$ and $\wt(b') < \wt(b)$.
\end{enumerate}
\end{theo}

The new basis $\Gj(\clB) := \{ \Gj(b) \mid b \in \clB \}$ thus constructed is called the $\jmath$-canonical basis of $(M,G(\clB))$.

\begin{prop}
We keep the notation in Theorem \ref{j-canonical}. Then, $(\clL,\clB)$ is a $\jmath$-crystal basis, $(\clL,M_{\A},\psij_M(\clL))$ is a balanced triple, and $\Gj(\clB)$ is the global $\jmath$-crystal basis associated to the balanced triple $(\clL,M_{\A},\psij_M(\clL))$ and the basis $\clB$.
\end{prop}

\begin{proof}
That $(\clL,\clB)$ is a $\jmath$-crystal basis has already been stated in Theorem \ref{restriction of crystal}. Let us prove the rest. By the property $(2)$ of Theorem \ref{j-canonical}, it is clear that $\clL$ (resp., $M_\A$) is spanned by $\Gj(\clB)$ over $\Ao$ (resp., $\A$). Also, by $(1)$ of Theorem \ref{j-canonical}, $\psij_M(\clL)$ is spanned by $\Gj(\clB)$ over $\Ai$. Hence, the canonical homomorphism $\clL \cap M_{\A} \cap \psij_M(\clL) \rightarrow \clL/q\clL$ is an isomorphism, and therefore, $(\clL,M_{\A},\psij_M(\clL))$ is balanced. Finally, by Lemma \ref{Uja in UA}, the $\UA$-module $M_\A$ is also a $\UjA$-module. This proves the proposition.
\end{proof}

%\begin{lem}\label{sub and quotient of j-global}
%Let $V,\clL,\clB,\ol{\ \cdot \ }$ be as above. Suppose that $V$ has a global basis (resp., $\jmath$-global basis) with an $\A$-lattice $V_{\A}$. Take a subset $\clB' \subset \clB$ and set $V',\clL',V'_\A$ to be the subspace spanned by $G(\clB')$ over $\Q(p,q),\Ao,\A$, respectively. If every element in $G(\clB')$ is fixed by $\ol{\ \cdot \ }$, and that $V'$ is a $\U$-submodule (resp., $\Uj$-submodule) of $V$, then, $V'$ and $V/V'$ have global bases (resp., $\jmath$-global bases).
%\end{lem}
%
%\begin{proof}
%By our assumption, $\ol{\clL'}$ is the $\Ai$-lattice of $V'$ spanned by $G(\clB')$. Then, the assertion follows from Lemma \ref{sub and quotient of bt}.
%\end{proof}

\section{Kazhdan-Lusztig bases}\label{KL bases}
The subsequent three sections are dedicated to prove the existence of a global $\jmath$-crystal basis and its ``dual'' basis for $L(\bflm)$, $\bflm \in \Pj$. In this section, we formulate variants of the Kazhdan-Lusztig bases following \cite{KL79}, \cite{Deo87}, and \cite{L}.

\subsection{Hecke algebra of type $B$}
Fix $d \in \Z_{> 0}$. Let $W = W_d$ be the Weyl group of type $B_d$ with simple reflections $S = \{ s_0,s_1,\ldots,s_{d-1} \}$ such that
$$
s_0s_1s_0s_1 = s_1s_0s_1s_0, \qu s_is_{i+1}s_i = s_{i+1}s_is_{i+1} \IF i \geq 1,\qu s_is_j = s_js_i \IF |i-j| > 1.
$$

\begin{defi}\normalfont
The Hecke algebra $\clH = \clH(W)$ associated to $W$ with unequal parameters $p,q$ is the associative algebra over $\AZ := \Z[p,p\inv,q,q\inv]$ generated by $\{ H_s \mid s \in S \}$ subject to the following relations:
\begin{itemize}
\item $(H_s - q_s\inv)(H_s + q_s) = 0$ for all $s \in S$, where $q_s = p$ if $s = s_0$ and $q_s = q$ otherwise.
\item $H_{s_0}H_{s_1}H_{s_0}H_{s_1} = H_{s_1}H_{s_0}H_{s_1}H_{s_0}$.
\item $H_{s_i}H_{s_{i+1}}H_{s_i} = H_{s_{i+1}}H_{s_i}H_{s_{i+1}}$ if $i \geq 1$.
\item $H_{s_i}H_{s_j} = H_{s_j}H_{s_i}$ if $|i-j| > 1$.
\end{itemize}
\end{defi}

We often write $H_i = H_{s_i}$. For each $w \in W$ with a reduced expression $w = s_{i_1} \cdots s_{i_l}$, the product $H_{i_1} \cdots H_{i_l}$ is independent of the choice of a reduced expression of $w$; we denote it by $H_w$. Similarly, $q_w := q_{s_{i_1}} \cdots q_{s_{i_l}}$ is well-defined.

Let $U,V$ be modules over $\AZ$. We say a $\Z$-linear map $f:U \rightarrow V$ is anti-linear if it satisfies $f(gu) = \ol{g} f(u)$ for all $g \in \AZ$ and $u \in U$. In the sequel, we will often use the following automorphisms, all of which are involutions, of $\clH$.
\begin{lem}\label{automorphisms on H}
\ 
\begin{enumerate}
\item There exists a unique anti-linear algebra automorphism $\ol{\ \cdot \ }$ of $\clH$ such that $\ol{H_w} = H_{w\inv}\inv$.
\item There exists a unique anti-linear algebra automorphism $\sgn$ of $\clH$ such that $\sgn(H_w) = (-1)^{\ell(w)}H_{w}$. Here, $\ell:W \rightarrow \Z_{\geq 0}$ denotes the length function on $W$.
\item There exists a unique $\AZ$-algebra anti-automorphism $(\cdot)^\flat$ of $\clH$ such that $H_w^\flat = H_{w\inv}$.
\end{enumerate}
Moreover, all of these automorphisms commute with each other.
\end{lem}

For $y,w \in W$, define $r_{y,w} \in \AZ$ by
\begin{align}
\ol{H_w} = \sum_{y \in W} r_{y,w} H_y. \nonumber
\end{align}
It is well-known and easily proved that $r_{w,w} = 1$ for all $w \in W$ and $r_{y,w} = 0$ unless $y \leq w$.

\subsection{Kazhdan-Lusztig bases}
Let us formulate the Kazhdan-Lusztig basis and the dual Kazhdan-Lusztig basis. Set
\begin{align}
\begin{split}
\AZ^+ &:= \AZ \cap \Ao = p\Z[p,q,q\inv] \oplus q\Z[q], \\
\AZ^- &:= \ol{\AZ^+} = p\inv\Z[p\inv,q,q\inv] \oplus q\inv \Z[q\inv].
\end{split} \nonumber
\end{align}

\begin{theo}[{\cite[Theorem 1.1]{KL79}}, {\cite[Theorem 5.2]{L}}]\label{KL-bases}
For each $w \in W$, there exists a unique $C_w \in \clH$ such that
\begin{enumerate}
\item $\overline{C_w} = C_w$.
\item $C_w = H_w + \sum_{y < w} c_{y,w} H_y$ for some $c_{y,w} \in \AZ^+$. Here, $<$ denotes the Bruhat order on $W$.
\end{enumerate}
\end{theo}

%\begin{rem}\normalfont
%In the equal parameter case (i.e. $\phi$ is a constant map), if we write $C_w = \sum_{y \leq w} q^{\ell(y,w)} P_{y,w}(q^{-2}) H_y$, then, $P_{y,w}$'s are known as the Kazhdan-Lusztig polynoimals.
%\end{rem}

Replacing $\AZ^+$ with $\AZ^-$, we see the following: For each $w \in W$, there exists a unique $D_w \in \clH$ such that
\begin{enumerate}
\item $\overline{D_w} = D_w$.
\item $D_w = H_w + \sum_{y < w} d_{y,w} H_y$ for some $d_{y,w} \in \AZ^-$.
\end{enumerate}

\begin{rem}\label{C-basis and D-basis}\normalfont
Noting that the automorphisms $\ol{\ \cdot \ }$ and $\sgn$ commute with each other, it is easy to verify that $D_w = (-1)^{\ell(w)} \sgn(C_w)$.
\end{rem}

It is obvious from the definitions that both $\{ C_w \mid w \in W \}$ and $\{ D_w \mid w \in W \}$ form $\AZ$-bases of $\clH$. We call the former the Kazhdan-Lusztig basis, and the latter the dual Kazhdan-Lusztig basis of $\clH$.

\subsection{Left cell representations}
Let us recall from \cite{KL79} the notion of left cells of $W$ and the associated left cell representations.
\begin{defi}\normalfont
Let $y,w \in W$.
\begin{enumerate}
\item $y \rightarrow_L w$ if the coefficient of $C_y$ in $C_s C_w$ expanded in the Kazhdan-Lusztig basis is nonzero for some $s \in S$.
\item $y \leq_L w$ if there exist $y = y_0,y_1,\ldots,y_l = w \in W$ such that $y_{i-1} \rightarrow_L y_i$.
\item $y \Lcell w$ if $y \leq_L w$ and $w \leq_L y$.
\item $y <_L w$ if $y \leq_L w$ and $y \not\Lcell w$.
\item Each equivalence class of $W/\Lcell$ is called a left cell of $W$. We denote by $L(W)$ the set of left cells of $W$.
\end{enumerate}
\end{defi}

\begin{rem}\normalfont
By Remark \ref{C-basis and D-basis}, we obtain the same equivalence relation as $\Lcell$ if we replace $C_w$'s by $D_w$'s.
\end{rem}

For each $X \in L(W)$ and $x \in X$, set
\begin{align}
&C_{\leq_L X} = \bigoplus_{y \leq_L x} \AZ C_y, \qu C_{<_L X} = \bigoplus_{y <_L x} \AZ C_y, \qu C^L_X = C_{\leq_L X} / C_{<_L X}, \nonumber\\
&D_{\leq_L X} = \bigoplus_{y \leq_L x} \AZ D_y, \qu D_{<_L X} = \bigoplus_{y <_L x} \AZ D_y, \qu D^L_X = D_{\leq_L X} / D_{<_L X}. \nonumber
\end{align}
Note that these are independent of the choice of $x \in X$. We denote the image of $m \in C_{\leq_L X}$ (resp., $m \in D_{\leq_L X}$) under the quotient map $C_{\leq_L X} \rightarrow C^L_X$ (resp., $D_{\leq_L X} \rightarrow D^L_X$) by $[m]_X$ (resp., $[m]'_X$).

\begin{lem}
Let $X \in L(W)$. Then, $C_{\leq_L X}$, $C_{<_L X}$, $D_{\leq_L X}$, and $D_{<_L X}$ are left ideals of $\clH$, and consequently, $C^L_{X}$ and $D^L_{X}$ are left $\clH$-modules. Moreover, $C^L_{X}$ has a basis $\{ [C_x]_X \mid x \in X \}$, while $D^L_{X}$ has a basis $\{ [D_x]'_X \mid x \in X \}$.
\end{lem}

\begin{proof}
The assertions are obvious from the definitions.
\end{proof}

We call $C^L_X$ the left cell representation of $\clH(W)$ associated to $X \in L(W)$.

\subsection{Bilinear form on $\clH$}
Let $\clH^* := \Hom_{\AZ}(\clH,\AZ)$. $\clH^*$ has a left $\clH$-module structure given by
\begin{align}
(Hf)(H') = f(H^\flat H'), \qu \Forall f \in \clH^*,\ H,H' \in \clH. \nonumber
\end{align}
Let $\{ h_w \mid w \in W \} \subset \clH^*$ be the dual basis of $\{ H_w \mid w \in W \}$, that is, they are characterized by $h_y(H_w) = \delta_{y,w}$ for all $y,w \in W$.

\begin{lem}\label{module structure of V*}
For each $w \in W$ and $s \in S$, the following holds.
\begin{align}
H_s h_w = \begin{cases}
h_{sw} \qu & \IF w < sw, \\
h_{sw} + (q_s\inv - q_s)h_w \qu & \IF sw < w.
\end{cases} \nonumber
\end{align}
\end{lem}

\begin{proof}
For each $y \in W$, we compute as
\begin{align}
\begin{split}
(H_sh_w)(H_y) &= h_w(H_sH_y) \\
&= \begin{cases}
h_w(H_{sy}) \qu & \IF sy > y, \\
h_w(H_{sy} + (q_s\inv-q_s)H_y) \qu & \IF sy < y
\end{cases} \\
&= \begin{cases}
1 \qu & \IF sy > y \AND sy = w, \\
1 \qu & \IF sy < y \AND sy = w, \\
q_s\inv - q_s \qu & \IF sy < y \AND y = w, \\
0 \qu & \OW
\end{cases} \\
&= \begin{cases}
h_{sw}(H_y) \qu & \IF sw > w, \\
(h_{sw} + (q_s\inv - q_s)h_w)(H_y) \qu & \IF sw < w.
\end{cases}
\end{split} \nonumber
\end{align}
This implies
\begin{align}
H_sh_w = \begin{cases}
h_{sw} \qu & \IF sw > w, \\
h_{sw} + (q_s\inv - q_s)h_w \qu & \IF sw < w.
\end{cases} \nonumber
\end{align}
Thus, the proof completes.
\end{proof}

There exists an anti-linear automorphism $\ol{\ \cdot \ }$ of $\clH^*$ defined by $\ol{f}(H) = \ol{f(\ol{H})}$ for $f \in \clH^*$, $H \in \clH$.

\begin{lem}\label{bar on V*}
For each $w \in W$, we have
\begin{align}
\ol{h_w} = \sum_{y \geq w} \ol{r_{w,y}} h_y. \nonumber
\end{align}
In particular, $\ol{h_{w_0}} = h_{w_0}$, where $w_0 \in W$ denotes the longest element.
\end{lem}

\begin{proof}
Let $y \in W$. Then, we have
\begin{align}
\ol{h_w}(H_y) = \ol{h_w(\ol{H_y})} = \ol{h_w(\sum_{z \leq y} r_{z,y} H_z)} = \ol{r_{w,y}}. \nonumber
\end{align}
Since $\ol{h_w} = \sum_{y \in W} \ol{h_w}(H_y) h_y$, the assertion follows.
\end{proof}

Let $\{ C_w^* \mid w \in W \} \subset \clH^*$ denote the dual basis of $\{C_w \mid w \in W \}$.

\begin{prop}\label{characterization of C^*}
$C_w^*$ is characterized by the following two conditions:
\begin{enumerate}
\item $\ol{C_w^*} = C_w^*$.
\item $C_w^* = h_w + \sum_{z > w} c_{w,z}^* h_z$ for some $c_{w,z}^* \in \AZ^+$.
\end{enumerate}
\end{prop}

\begin{proof}
Thanks to Lemma \ref{bar on V*}, one can prove that there exists a unique $C'_w \in \clH^*$ such that $\ol{C'_w} = C'_w$ and $C'_w - h_w \in \bigoplus_{y > w} \AZ^+ h_y$ in a similar way to Theorem \ref{KL-bases}. Hence, it suffices to show that $C_w^*$ satisfies the two conditions.

The first condition is verified as follows. For each $y \in W$, we have
\begin{align}
\ol{C_w^*}(C_y) = \ol{C_w^*(\ol{C_y})} = \ol{C_w^*(C_y)} = \ol{\delta_{y,w}} = \delta_{y,w} = C_w^*(C_y). \nonumber
\end{align}
Since $\{ C_y \mid y \in W \}$ is a basis of $\clH$, we obtain $\ol{C_w^*} = C_w^*$.

Next, we prove the second condition. For each $y \in W$, we can write $H_y = C_y + \sum_{z < y} b_{z,y}C_z$ for some $b_{z,y} \in \AZ^+$. Then, we have
\begin{align}
C_w^* = \sum_{y \in W} C_w^*(H_y)h_y = h_w + \sum_{y > w} b_{w,y} h_y. \nonumber
\end{align}
This completes the proof.
\end{proof}

\begin{lem}\label{Duality of Hecke}
The linear map $d\index{dHeckealgebra@$d$ (Hecke algebra)}:\clH \rightarrow \clH^*;\ H \mapsto H \cdot h_{w_0}$ gives an isomorphism of left $\clH$-modules. Moreover, we have
\begin{enumerate}
\item $d(\ol{H_y}) = h_{yw_0}$ for all $y \in W$.
\item $d(\ol{H}) = \ol{d(H)}$ for all $H \in \clH$.
\end{enumerate}
\end{lem}

\begin{proof}
By Lemma \ref{module structure of V*}, the linear map $\vphi:\clH \rightarrow \clH^*;\ H_w \mapsto h_w$ is an isomorphism of left $\clH$-modules. On the other hand, the map $\psi:\clH \rightarrow \clH;\ H \mapsto H \cdot H_{w_0}$ is clearly an isomorphism of left $\clH$-modules. Thus, the composite map $d:= \vphi \circ \psi: \clH \rightarrow \clH^*$ is an isomorphism of left $\clH$-modules satisfying
\begin{align}
d(H) = \vphi(H \cdot H_{w_0}) = H \cdot \vphi(H_{w_0}) = H \cdot h_{w_0} \qu \Forall H \in \clH. \nonumber
\end{align}
Also, we have, for all $y \in W$,
\begin{align}
d(\ol{H_y}) = \vphi(\ol{H_y} \cdot H_{w_0}) = \vphi(H_{y\inv}\inv \cdot H_{y\inv} H_{yw_0}) = \vphi(H_{yw_0}) = h_{yw_0}. \nonumber
\end{align}
Finally, for each $H,H' \in \clH$, we have
\begin{align}
\begin{split}
&d(\ol{H})(H') = (\ol{H}\cdot h_{w_0})(H') = h_{w_0}\left( \left( \ol{H} \right)^\flat H' \right), \\
&\ol{d(H)}(H') = \ol{d(H)(\ol{H'})} = \ol{h_{w_0}(H^\flat \ol{H'})} = \ol{h_{w_0}} \left( \ol{H^\flat} H' \right).
\end{split} \nonumber
\end{align}
Then, the equality $d(\ol{H}) = \ol{d(H)}$ follows from the facts that $\ol{h_{w_0}} = h_{w_0}$ and $\left( \ol{H} \right)^\flat = \ol{H^\flat}$; the former is proved in Lemma \ref{bar on V*}, and the latter is in Lemma \ref{automorphisms on H}.
\end{proof}

Using this isomorphism, we define a bilinear form $\la \cdot \mid \cdot \ra$ on $\clH$ by
\begin{align}
\la H \mid H' \ra := d(H')(H), \qu (H,H' \in \clH). \nonumber
\end{align}
Clearly, this bilinear form satisfies $\la H' \mid HH'' \ra = \la H^\flat H' \mid H'' \ra$ for all $H,H',H'' \in \clH$.

\begin{lem}
The bilinear from $\la \cdot \mid \cdot \ra$ is symmetric.
\end{lem}

\begin{proof}
Let $H_1,H_2 \in \clH$. It suffices to show that $h_{w_0}(H_2^\flat H_1) = h_{w_0}(H_1^\flat H_2)$. Since $H_{w_0}^\flat = H_{w_0}$, it holds that $h_{w_0}(H^\flat) = h_{w_0}(H)$ for all $H \in \clH$. Then, the assertion follows if one notes $(H_2^\flat H_1)^\flat = H_1^\flat H_2$.
\end{proof}

\begin{prop}\label{duality of KL-bases}
The bases $\{ C_w \mid w \in W \}$ and $\{ D_{ww_0} \mid w \in W \}$ are dual to each other with respect to $\la \cdot \mid \cdot \ra$, that is, we have $\la C_y \mid D_w \ra = \delta_{y,ww_0}$ for all $y,w \in W$.
\end{prop}

\begin{proof}
Recall that $D_w = \sum_{y \leq w} d_{y,w} H_y$ with $d_{w,w} = 1$ and $d_{y,w} \in \AZ^-$ for all $y < w$. Then, we have
\begin{align}
\begin{split}
&\ol{d(D_w)} = d(\ol{D_w}) = d(D_w), \\
&d(D_w) = d(\ol{D_w}) = d(\sum_{y \leq w} \ol{d_{y,w}} \ol{H_y}) = \sum_{y \leq w} \ol{d_{y,w}} h_{yw_0} = \sum_{z \geq ww_0} \ol{d_{zw_0,w}} h_{z}.
\end{split} \nonumber
\end{align}
This and Proposition \ref{characterization of C^*} show that $d(D_w) = C^*_{ww_0}$. Hence, it holds that $\la C_y \mid D_w \ra = C^*_{ww_0}(C_y) = \delta_{y,ww_0}$, which proves the proposition.
\end{proof}

Here, we describe the duality between $C^L_X$'s and $D^L_X$'s.

\begin{lem}\label{left cell and w_0}
Let $y,w \in W$, $X \in L(W)$. Then, the following hold.
\begin{enumerate}
\item $y \rightarrow_L w$ if and only if $ww_0 \rightarrow_L yw_0$.
\item $y \leq_L w$ if and only if $ww_0 \leq_L yw_0$.
\item $Xw_0 := \{ xw_0 \mid x \in X \} \in L(W)$.
\end{enumerate}
\end{lem}

\begin{proof}
We first prove part $(1)$. Suppose that $y \rightarrow_L w$. Then, there exists $s \in S$ such that $\la C_s C_w \mid D_{yw_0} \ra \neq 0$. This implies that $\la C_w \mid C_s D_{yw_0} \ra \neq 0$, and hence, we obtain $ww_0 \rightarrow_L yw_0$. Replacing $y,w$ by $yw_0,ww_0$, we also have the opposite indication. This proves part $(1)$. Assertion $(2)$ is an immediate consequence of $(1)$. We prove part $(3)$. Let $x \in X$. Then, $X = \{ y \in W \mid x \leq_L y \leq_L x \}$. By part $(2)$, we have $x \leq_L y \leq_L x$ if and only if $xw_0 \leq_L yw_0 \leq_L xw_0$. This implies that  $Xw_0 = \{ z \in W \mid xw_0 \leq_L z \leq_L xw_0 \}$, and it is a unique left cell of $W$ containing $xw_0$. Thus, the proof completes.
\end{proof}

\begin{lem}
The bilinear from $\la \cdot \mid \cdot \ra$ induces a non-degenerate bilinear form on $C^L_X \times D^L_{Xw_0}$. Moreover, $\{ [C_x]_X \mid x \in X \}$ and $\{ [D_{xw_0}]_{Xw_0}' \mid x \in X \}$ form bases which are dual to each  other.
\end{lem}

\begin{proof}
Let $x \in X$, $y,w \in W$ be such that $y <_L x$ and $ww_0 <_L xw_0$. It suffices to show that $\la C_y \mid D_u \ra = 0$ for all $u \leq_L xw_0$ and $\la C_v \mid D_{ww_0} \ra = 0$ for all $v \leq_L x$. Both are obvious from Lemma \ref{left cell and w_0} $(2)$.
\end{proof}

\begin{prop}
Let $X \in L(W)$. Then, we have an isomorphism $D^L_{Xw_0} \simeq C^L_X$ of $\clH$-modules.
\end{prop}

\begin{proof}
It suffices to show that the characters $\ch_{D^L_{Xw_0}}$ of $D^L_{Xw_0}$ and $\ch_{C^L_X}$ of $C^L_X$ coincide with each other. For each $w \in W$, we compute as
\begin{align}
\begin{split}
\ch_{C^L_X}(H_w) &= \sum_{x \in X} \la H_w [C_x]_X \mid [D_{xw_0}]'_{Xw_0} \ra \\
&= \sum_{x \in X} \la [C_x]_X \mid H_{w\inv} [D_{xw_0}]_{Xw_0}' \ra \\
&= \ch_{D^L_{Xw_0}}(H_{w\inv}) = \ch_{D^L_{Xw_0}}(H_w).
\end{split} \nonumber
\end{align}
Thus, the proof completes.
\end{proof}

\subsection{Parabolic Kazhdan-Lusztig bases}\label{Parabolic Kazhdan-Lusztig bases}
Throughout this subsection, we fix a subset $J \subset \{ 0,1,\ldots,d-1 \}$ arbitrarily. Let $W_J$ denote the parabolic subgroup of $W$ generated by $\{ s_j \mid j \in J \}$, ${}^JW$ the set of minimal length coset representatives for $W_J \backslash W$, and $w_J \in W_J$ the longest element. Also, we set
\begin{align}
x_J := q_{w_J} \sum_{w \in W_J} q_w\inv H_w \in \clH. \nonumber
\end{align}

\begin{lem}\label{property of x_J}
Let $j \in J$. Then, the following hold.
\begin{enumerate}
\item $x_J H_j = q_{s_j}\inv x_J$.
\item $x_J^\flat = x_J$.
\item $x_J = C_{w_J}$. In particular, $\ol{x_J} = x_J$.
\end{enumerate}
\end{lem}

\begin{proof}
The assertion $(1)$ follows from a direct calculation and the fact that $W_J = \{ w \in W_J \mid w < s_jw \} \sqcup \{ w \in W_J \mid s_jw < w \}$. The assertion $(2)$ follows from the definition of $x_J$ and the facts that $W_J = \{ w\inv \mid w \in W_J \}$, and $q_{w\inv} = q_w$ for all $w \in W$. The proof of $(3)$ can be found in \cite[Proposition 1.17 (2)]{X94}.
\end{proof}

By Lemma \ref{property of x_J} $(1)$, the right ideal $x_J \clH$ of $\clH$ has a basis $\{ x_J H_w \mid w \in {}^JW \}$. Also, by Lemma \ref{property of x_J} $(3)$, $x_J \clH$ is closed under the involution $\ol{\ \cdot \ }$. Hence, we can construct analogs of the Kazhdan-Lusztig basis and the dual Kazhdan-Lusztig basis of $\clH$ in the ideal $x_J \clH$:

\begin{theo}{\cite[Proposition 3.2]{Deo87}}\label{parabolic KL basis}
\ 
\begin{enumerate}
\item For each $w \in {}^JW$, there exists a unique ${}^JC_w \in x_J \clH$ such that
\begin{enumerate}
\item $\ol{{}^JC_w} = {}^JC_w$.
\item ${}^JC_w = x_J(H_w + \sum_{\substack{y \in {}^JW \\ y < w}} {}^Jc_{y,w} H_y)$ for some ${}^Jc_{y,w} \in \AZ^+$.
\end{enumerate}
\item For each $w \in {}^JW$, there exists a unique ${}^JD_w \in x_J \clH$ such that
\begin{enumerate}
\item $\ol{{}^JD_w} = {}^JD_w$.
\item ${}^JD_w = x_J(H_w + \sum_{\substack{ y \in {}^JW \\ y < w}} {}^Jd_{y,w} H_y)$ for some ${}^Jd_{y,w} \in \AZ^-$.
\end{enumerate}
\end{enumerate}
\end{theo}

Clearly, $\{ {}^JC_w \mid w \in {}^JW \}$ and $\{ {}^JD_w \mid w \in {}^JW \}$ are linear bases of $x_J \clH$. We call them the parabolic Kazhdan-Lusztig basis and the dual parabolic Kazhdan-Lusztig basis of $x_J \clH$, respectively.

\begin{prop}{\cite[Proposition 3.4]{Deo87}}\label{JCw}
Let $w \in {}^JW$. Then, ${}^JC_w = C_{w_Jw}$.
\end{prop}

%\begin{proof}
%We have
%\begin{align}
%{}^JC_w &= x_J \sum_{\substack{y \in {}^JW \\ y \leq w}} {}^Jc_{y,w}H_y = \sum_{\substack{y \in {}^JW \\ y \leq w}} \sum_{x \in W_J} q_{w_J}q_x\inv {}^Jc_{y,w} H_{xy}. \nonumber
%\end{align}
%This shows that ${}^JC_w - H_{w_Jw} \in \bigoplus_{z < w_Jw} \AZ^+ H_z$. Hence, by Theorem \ref{KL-bases}, ${}^JC_w$ coincides with $C_{w_Jw}$.
%\end{proof}

\begin{prop}\label{JDw}
Let $w \in {}^JW$. Then, ${}^JD_w = x_J D_w$.
\end{prop}

\begin{proof}
For each $y \in W$, define $y_J \in W_J$ and ${}^Jy \in {}^JW$ to be the unique elements satisfying $y = y_J {}^Jy$ and $\ell(y) = \ell(y_J) + \ell({}^Jy)$. Then, we have
\begin{align}
\begin{split}
x_J D_w &= x_J \sum_{y \leq w} d_{y,w} H_y \\
&= x_J\left( H_w + \sum_{y < w} d_{y,w} H_{y_J}H_{{}^Jy} \right) \\
&= x_J\left( H_w + \sum_{y < w} q_{y_J}\inv d_{y,w} H_{{}^Jy} \right) \qu (\text{by Lemma \ref{property of x_J} $(1)$})\\
&= x_J\left( H_w + \sum_{\substack{y \in {}^JW \\ y < w}} \sum_{\substack{x \in W_J \\ xy < w}} q_x\inv d_{xy,w} H_y \right).
\end{split} \nonumber
\end{align}
This shows that $x_JD_w - x_JH_w \in \bigoplus_{\substack{y \in {}^JW \\ y < w}} \AZ^- x_JH_y$. Hence, by Theorem \ref{parabolic KL basis} $(2)$, $x_JD_w$ coincides with ${}^JD_w$.
\end{proof}

For a later use, let us consider $x_J C_y$ and $x_J D_y$ for $y \in W$.

\begin{prop}\label{from KL to parabolic KL}
Let $y \in W$. Then, we have
\begin{align}
x_J C_y = \sum_{\substack{w \in {}^JW \\ w_Jw \leq_L y}} \alpha_w {}^JC_w, \nonumber
\end{align}
for some $\alpha_w \in \AZ$.
\end{prop}

\begin{proof}
Let us write
\begin{align}
x_J C_y = \sum_{w \in {}^JW} \alpha_w {}^JC_w = \sum_{w \in {}^JW} \alpha_w C_{w_Jw} \qu \Forsome \alpha_w \in \AZ. \nonumber
\end{align}
Also, by the definition of $\leq_L$, we can write
\begin{align}
x_J C_y = \sum_{z \leq_L y} \beta_z C_z \qu \Forsome \beta_z \in \AZ. \nonumber
\end{align}
This shows $\alpha_w = 0$ unless $w_Jw \leq_L y$.
\end{proof}

\begin{lem}{\cite[Theorem 6.6 $(b)$]{L}}\label{HsDw}
Let $w \in W$ and $s \in S$ be such that $sw < w$. Then, it holds that $H_s D_w = -q_s D_w$.
\end{lem}

\begin{prop}
Let $y \in W \setminus {}^JW$. Then, $x_J D_y = 0$.
\end{prop}

\begin{proof}
Since $y \notin {}^JW$, there exists $j \in J$ such that $s_jy < y$. For such $j$, we have $x_J H_j = q_j\inv x_J$ (Lemma \ref{property of x_J} (1)) and $H_j D_y = -q_j D_y$ (Lemma \ref{HsDw}). Hence, we obtain
\begin{align}
x_J D_y = q_j x_J H_j D_y = -q_j^2 x_J D_, \nonumber
\end{align}
which implies $x_J D_y = 0$, as desired.
\end{proof}

Set $P_J := q_{w_J} \sum_{x \in W_J} q_x^{-2} \in \AZ$. Note that, by Lemma \ref{property of x_J} $(1)$, it holds that $x_J^2 = P_J x_J$. Then, for each $H,H' \in \clH$, we have
\begin{align}
\la x_JH \mid x_JH' \ra = \la x_J^2 H \mid H' \ra = P_J \la H \mid H' \ra \in P_J \AZ; \nonumber
\end{align}
here, we use Lemma \ref{property of x_J} $(2)$. Hence, we can define a $\AZ$-valued bilinear form $\la \cdot \mid \cdot \ra_J$ on $x_J \clH$ by $\la \cdot \mid \cdot \ra_J := \frac{1}{P_J} \la \cdot \mid \cdot \ra$.

\begin{prop}\label{dual of parabolic KL}
The basis $\{ {}^JC_w \mid w \in {}^JW \}$ and $\{ {}^JD_{w_Jww_0} \mid w \in {}^JW \}$ are dual to each other with respect to $\la \cdot \mid \cdot \ra_J$, that is, we have $\la {}^JC_y \mid {}^JD_{w} \ra_J = \delta_{y,w_Jww_0}$ for all $y,w \in {}^JW$.
\end{prop}

\begin{proof}
Let $y,w \in {}^JW$. We compute as follows:
\begin{align}
\begin{split}
\la {}^JC_y \mid {}^JD_w \ra_J &= \frac{1}{P_J} \la {}^JC_y \mid {}^JD_w \ra \\
&= \frac{1}{P_J} \la C_{w_Jy} \mid x_J D_w \ra \qu (\text{by Proposition \ref{JCw} and \ref{JDw}}) \\
&= \la C_{w_Jy} \mid D_w \ra \qu (\text{since $C_{w_Jy} = {}^JC_y \in x_J \clH$}) \\
&= \delta_{w_Jy,ww_0} = \delta_{y,w_Jww_0} \qu (\text{by Proposition \ref{duality of KL-bases}}).
\end{split} \nonumber
\end{align}
This proves the proposition.
\end{proof}

\section{Hecke modules and their centralizers}\label{section3}
\subsection{Fundamental properties}
We follow ideas in \cite[Chapter 9.1]{DDPW08}. Let $\pi$ be an index set. Suppose that we are given a map $\pi \rightarrow \{ J \mid J \subset \{ 0,1,\ldots,d-1 \} \}$. We denote by $I_\lambda$ the image of $\lambda \in \pi$ under this map. For each $\lambda \in \pi$, for simplicity, we will denote $W_{I_\lambda}$, $w_{I_\lambda}$, $x_{I_\lambda}$, etc. by $W_\lambda$, $w_{\lambda}$, $x_{\lambda}$, etc.

\begin{defi}\normalfont
Associated with $\pi$, we define a right $\Hc$-module $\bbT(\pi) := \bigoplus_{\lambda \in \pi} x_\lambda \Hc$, and its centralizer algebra $\bbS(\pi) := \End_\Hc(\bbT(\pi))$; we let $\bbS(\pi)$ act on $\bbT(\pi)$ from the left.
\end{defi}

It is obvious that $\bbT(\pi)$ has two bases $\{ {}^\lm C_w \mid \lm \in \pi,\ w \in {}^\lm W \}$ and $\{ {}^\lm D_w \mid \lm \in \pi,\ w \in {}^\lm W \}$; we call them the Kazhdan-Lusztig basis and dual Kazhdan-Lusztig basis, respectively.

For each $m = \sum_{\lambda \in \pi} m_\lambda \in \bbT(\pi)$ with $m_\lambda \in x_\lambda \Hc$, we define $\ol{m} \in \bbT(\pi)$ to be $\sum_{\lambda \in \pi} \ol{m_\lambda}$. Also, for each $f \in \bbS(\pi)$, define $\ol{f} \in \bbS(\pi)$ by $\ol{f}(m) = \ol{f(\ol{m})}$ for all $m \in \bbT(\pi)$. This gives anti-linear automorphisms $\ol{\ \cdot \ }$ on $\bbT(\pi)$ and $\bbS(\pi)$.

For each $\lambda \in \pi$, define $p_\lambda\index{plambda@$p_\lambda$} \in \bbS(\pi)$ to be the composite\begin{align}
p_\lambda: \bbT(\pi) \twoheadrightarrow x_\lambda \Hc \hookrightarrow \bbT(\pi) \nonumber
\end{align}
of the projection and the inclusion. Clearly, $\{ p_\lambda \mid \lambda \in \pi \}$ is a family of orthogonal idempotents with $\sum_{\lambda \in \pi} p_\lambda = \id_{\bbT(\pi)}$. Hence, we have a decomposition
\begin{align}
\bbS(\pi) = \bigoplus_{\lambda,\mu \in \pi} p_\lambda \bbS(\pi) p_\mu, \qu p_\lambda \bbS(\pi) p_\mu = \Hom_\Hc(x_\mu \Hc, x_\lambda \Hc). \nonumber
\end{align}
Take $f \in \Hom_\Hc(x_\mu \Hc, x_\lambda \Hc)$ arbitrarily. Since $x_\mu \Hc$ is generated (as a right $\Hc$-module) by $x_\mu$, the $f$ is determined by $f(x_\mu) \in x_\lambda \Hc$. Let us write
\begin{align}
f(x_\mu) = \sum_{w \in {}^\lambda W} c_{\lambda,w,\mu}(f) x_\lambda H_w, \qu \Forsome c_{\lambda,w,\mu} \in \AZ. \nonumber
\end{align}

\begin{lem}
Let $w \in {}^\lambda W$ and $j \in I_\mu$ be such that $w < ws_j$. Then, we have
\begin{align}
c_{\lambda,w,\mu}(f) = q_j c_{\lambda,ws_j,\mu}(f). \nonumber
\end{align}
Consequently, we have
\begin{align}
f(x_\mu) = \sum_{w \in {}^\lambda W^\mu} \sum_{\substack{y \in W_\mu \\ wy \in {}^\lambda W}} q_y\inv c_{\lambda,w,\mu}(f) x_\lambda H_{wy}, \nonumber
\end{align}
and hence, $f$ is determined by $(c_{\lambda,w,\mu}(f))_{w \in {}^\lambda W^\mu} \in \AZ^{{}^\lm W^\mu}$, where ${}^\lambda W^\mu := {}^\lambda W \cap ({}^\mu W)\inv$.
\end{lem}

\begin{proof}
We have
\begin{align}
\begin{split}
q_j\inv f(x_\mu) &= f(x_\mu H_j) \\
&= f(x_\mu)H_j \\
&= \sum_{\substack{w \in {}^\lambda W \\ ws_j < w}} c_{\lambda,w,\mu}(f) x_\lambda(H_{ws_j} + (q_j\inv - q_j) H_w) + \sum_{\substack{w \in {}^\lambda W \\ ws_j > w}} c_{\lambda,w,\mu}(f) x_\lambda H_{ws_j} \\
&= \sum_{\substack{w \in {}^\lambda W \\ ws_j < w}}(c_{\lambda,ws_j,\mu}(f) + (q_j\inv - q_j)c_{\lambda,w,\mu}(f))x_\lambda H_w + \sum_{\substack{w \in {}^\lambda W \\ ws_j > w}} c_{\lambda,ws_j,\mu}(f) x_\lambda H_w.
\end{split} \nonumber
\end{align}
Comparing the coefficients of $x_\lambda H_w$, we obtain the assertion.
\end{proof}

Conversely, given $(c_{\lambda,w,\mu})_{w \in {}^\lambda W^\mu} \in \AZ^{{}^\lm W^\mu}$, there exists a unique $g \in \Hom_\Hc(x_\mu \Hc, x_\lambda \Hc)$ such that $c_{\lambda,w,\mu}(g) = c_{\lambda,w,\mu}$ for all $w \in {}^\lambda W^\mu$. Thus, we obtain an $\AZ$-linear isomorphism between $\AZ^{{}^\lambda W^\mu}$ and $\Hom_{\Hc}(x_\mu \Hc,x_\lambda \Hc)$.

\begin{lem}[{\cite[Theorem 4.18]{DDPW08}}]\label{double coset}
Let $\lm,\mu \in \pi$. For each $x \in {}^\lm W^\mu$, there exists a unique $J_x \subset \{ 0,1,\ldots,d-1 \}$ such that the multiplication map
$$
W_\lm \times \{x\} \times {}^{J_x} W_\mu \rightarrow W_\lm x W_\mu;\ (u,x,v) \mapsto uxv
$$
is a bijection, where ${}^{J_x} W_\mu := {}^{J_x} W \cap W_\mu$. Moreover, we have $\ell(uxv) = \ell(u) + \ell(x) + \ell(v)$ for all $u \in W_\lm$ and $v \in {}^{J_x} W_\mu$.
\end{lem}

For $\lambda, \mu \in \pi$ and $x \in {}^\lambda W^\mu$, define $\xi_{\lambda,x,\mu} \in \Hom_{\Hc}(x_\mu \Hc,x_\lambda \Hc)$ to be the one corresponding to $(\delta_{x,w} q_{x'})_{w \in {}^\lambda W^\mu} \in \AZ^{{}^\lambda W^\mu}$, where $x' \in W_\mu$ is the longest element in ${}^{J_x}W_\mu$ ($J_x$ is as in Lemma \ref{double coset}). Then, the next proposition is clear.

\begin{prop}
$\{ \xi_{\lambda,x,\mu} \mid \lambda,\mu \in \pi,\ x \in {}^\lambda W^\mu \}$ forms a basis of $\bbS(\pi)$.
\end{prop}

For each $\lambda,\mu \in \pi$, $x \in {}^\lambda W^\mu$, set
\begin{align}
\eta_{\lambda,x,\mu} = q_{w_\lambda xx'} \sum_{w \in W_\lambda x W_\mu} q_w\inv H_w. \nonumber
\end{align}

\begin{lem}\label{property of eta}
Let $\lambda,\mu \in \pi$, $x \in {}^\lambda W^\mu$.
\begin{enumerate}
%\item $\eta_{\lambda,x,\mu} \in x_\lambda \Hc x_\mu$.
%\item $\eta_{\lambda,x,\mu} \in \Q(\Gamma) x_\lambda H_x x_\mu$.
\item $\eta_{\lambda,x,\mu}^\flat = \eta_{\mu,x\inv,\lambda}$.
\item $\xi_{\lambda,x,\mu}(x_\mu) = \eta_{\lambda,x,\mu} = \frac{1}{P_\mu}\eta_{\lambda,x,\mu} \cdot x_\mu = \frac{1}{P_\lm} x_\lm \cdot \eta_{\lm,x,\mu}$.
\item $\ol{\xi_{\lambda,e,\mu}} = \xi_{\lambda,e,\mu}$, where $e$ denotes the identity element of $W$.
\end{enumerate}
\end{lem}

\begin{proof}
\begin{enumerate}
\item By the definition of $x'$, we have $y := w_\lambda x x'$ is the longest element in $W_\lm x W_\mu$. Also, it is easily checked that the map $W \rightarrow W,\ w \mapsto w\inv$ gives a bijection $W_\lambda x W_\mu \rightarrow W_\mu x\inv W_\lambda$. Since this bijection preserves the length, $y\inv$ is the longest element in $W_\mu x\inv W_\lm$. Then, we compute $\eta_{\lambda,x\mu}^\flat$ as follows:
\begin{align}
\eta_{\lambda,x\mu}^\flat = q_y \sum_{w \in W_\lambda x W_\mu} q_w\inv H_{w\inv} = q_{y} \sum_{w \in W_\mu x\inv W_\lambda} q_{w\inv}\inv H_w = q_{y\inv} \sum_{w \in W_\mu x\inv W_\lambda} q_w\inv H_w = \eta_{\mu,x\inv,\lambda}. \nonumber
\end{align}
\item By the definition of $\xi_{\lambda,x,\mu}$, we have
\begin{align}
\begin{split}
\xi_{\lambda,x,\mu}(x_\mu) &= \sum_{\substack{y \in W_\mu \\ xy \in {}^\lambda W}} q_{x'}q_y\inv x_\lambda H_{xy} \\
&=  \sum_{\substack{y \in W_\mu \\ xy \in {}^\lambda W}} \sum_{z \in W_\lambda} q_{x'}q_y\inv q_{w_\lambda} q_z\inv H_{zxy} \qu (\text{by the definition of $x_\lambda$})\\
&= \sum_{w \in W_\lambda x W_\mu} q_{w_\lambda} q_x q_{x'} q_w\inv H_w = \eta_{\lambda,x,\mu}.
\end{split} \nonumber
\end{align}
This proves the first equation. Next, we have
\begin{align}
\eta_{\lm,x,\mu} \cdot x_\mu = \xi_{\lm,x,\mu}(x_\mu^2) = P_\mu \xi_{\lm,x,\mu}(x_\mu) = P_\mu \eta_{\lm,x,\mu}, \nonumber
\end{align}
which implies the second equality. Finally, the third equality follows from the fact that $\eta_{\lm,x,\mu} = \xi_{\lm,x,\mu}(x_\mu) \in x_\lm \clH$.
\item It suffices to check that $\ol{\xi_{\lambda,e,\mu}}(x_\nu) = \xi_{\lambda,e,\mu}(x_\nu)$ for all $\nu \in \pi$. Only the non-trivial case is when $\nu = \mu$. Since we have
\begin{align}
\ol{\xi_{\lambda,e,\mu}}(x_\mu) = \ol{\xi_{\lambda,e,\mu}(x_\mu)} = \ol{\eta_{\lambda,e,\mu}}, \nonumber
\end{align}
the problem is reduced to proving that $\eta_{\lambda,e,\mu}$ is fixed under the involution $\ol{\ \cdot \ }$. One can write
\begin{align}
\eta_{\lambda,e,\mu} = \sum_{w \in W_\lambda W_\mu} q_{w_\lambda} q_{e'} q_w\inv H_w = x_\lm \sum_{y \in {}^\lm W_\mu} q_{e'} q_y\inv H_y. \nonumber
\end{align}
On the other hand, we have
$$
x_\lm x_\mu = x_\lm x_{I_\lm \cap I_\mu} \sum_{y \in {}^\lm W_\mu} q_{e'} q_y\inv H_y = P_{I_\lm \cap I_\mu} x_\lm \sum_{y \in {}^\lm W_\mu} q_{e'} q_y\inv H_y.
$$
Hence, we obtain
$$
\eta_{\lm,e,\mu} = \frac{1}{P_{I_\lm \cap I_\mu}} x_\lm x_\mu,
$$
which is invariant under $\ol{\ \cdot \ }$. Thus, the proof completes.
\end{enumerate}
\end{proof}

\begin{prop}\label{flat on bbS}
The linear map $\flat:\bbS(\pi) \rightarrow \bbS(\pi);\ \xi_{\lambda,x,\mu} \mapsto \xi_{\mu,x\inv,\lambda}$ defines an $\AZ$-algebra anti-automorphism on $\bbS(\pi)$.
\end{prop}

\begin{proof}
We have to verify that $(\xi_{\lambda,x,\mu} \cdot \xi_{\kappa,y,\nu})^{\flat} = \xi_{\nu,y\inv,\kappa} \cdot \xi_{\mu,x\inv,\lambda}$ for all $\lambda,\mu,\nu,\kappa$ and $x \in {}^\lambda W^{\mu}$, $y \in {}^\kappa W^\nu$. Since the both sides are equal to zero unless $\kappa = \mu$, we may assume that $\kappa = \mu$. Let us write
\begin{align}\label{prod of xi}
&\xi_{\lambda,x,\mu} \cdot \xi_{\mu,y,\nu} = \sum_{z \in {}^\lambda W^\nu} c_z \xi_{\lambda,z,\nu} \qu \Forsome c_z \in \AZ.
\end{align}
Applying the both sides to $x_\nu \in \bbT(\pi)$, by Lemma \ref{property of eta} $(2)$, we obtain
\begin{align}\label{prod of eta}
\frac{1}{P_\mu} \eta_{\lambda,x,\mu} \eta_{\mu,y,\nu} = \sum_{z \in {}^\lambda W^\mu} c_z \eta_{\lambda,z,\nu}.
\end{align}

To prove the assertion, we compute as follows:
\begin{align}
\begin{split}
\xi_{\nu,y\inv,\mu} \cdot \xi_{\mu,x\inv,\lambda}(x_\lambda) &= \frac{1}{P_\mu}\eta_{\nu,y\inv,\mu} \cdot \eta_{\mu, x\inv,\lambda} \qqu (\text{by Lemma \ref{property of eta} $(2)$})\\
&= (\frac{1}{P_\mu}\eta_{\lambda,x,\mu} \cdot \eta_{\mu,y,\nu})^\flat \qqu (\text{by Lemma \ref{property of eta} $(1)$})\\
&= (\sum_{z \in {}^\lambda W^\nu} c_z \eta_{\lambda,z,\nu})^\flat \qqu (\text{by equation \eqref{prod of eta}})\\
&= \sum_{z \in {}^\lambda W^\nu} c_z \eta_{\nu,z\inv,\lambda} \qqu (\text{by Lemma \ref{property of eta} $(1)$})\\
&= \sum_{z \in {}^\lambda W^\nu} c_z \xi_{\nu,z\inv,\lambda}(x_\lambda) \qqu (\text{by Lemma \ref{property of eta} $(2)$})\\
&= (\xi_{\lambda,x,\mu} \cdot \xi_{\mu,y,\nu})^\flat(x_\lambda) \qqu (\text{by equation \eqref{prod of xi}}).
\end{split} \nonumber
\end{align}
This shows that $\xi_{\nu,y\inv,\mu} \cdot \xi_{\mu,x\inv,\lambda} = (\xi_{\lambda,x,\mu} \cdot \xi_{\mu,y,\nu})^\flat$, and hence, the proof completes.
\end{proof}

Recall the bilinear form $\la \cdot \mid \cdot \ra_\lambda = \la \cdot \mid \cdot \ra_{I_\lm}$ on $x_\lambda \Hc$ defined in Section \ref{Parabolic Kazhdan-Lusztig bases}.

\begin{prop}
Let $\lambda,\mu \in \pi$, $m \in x_\lambda \Hc$, and $n \in x_\mu \Hc$. Then, for each $w \in {}^\lambda W^\mu$, we have
\begin{align}
\la m \mid \xi_{\lambda,w,\mu}(n) \ra_\lambda = \la \xi_{\lambda,w,\mu}^\flat(m) \mid n \ra_\mu. \nonumber
\end{align}
\end{prop}

\begin{proof}
We compute as follows:
\begin{align}
\begin{split}
\la m \mid \xi_{\lambda,w,\mu}(n) \ra_\lambda
&= \frac{1}{P_\lambda} \la m \mid \xi_{\lambda,w,\mu}(n) \ra \\
&= \frac{1}{P_\lambda P_\mu} \la m \mid \eta_{\lambda,w,\mu} n \ra  \qqu (\text{by Lemma \ref{property of eta} $(2)$})\\
&= \frac{1}{P_\lambda P_\mu} \la \eta_{\mu,w\inv,\lambda} m \mid n \ra \qqu (\text{by Lemma \ref{property of eta} $(1)$})\\
&= \frac{1}{P_\mu} \la \xi_{\mu,w\inv,\lambda}(m) \mid n \ra \qqu (\text{by Lemma \ref{property of eta} $(2)$})\\
&= \la \xi_{\lambda,w,\mu}^\flat(m) \mid n \ra_\mu. \nonumber
\end{split}
\end{align}
This proves the proposition.
\end{proof}

Define a bilinear form $\la \cdot \mid \cdot \ra_\pi$ on $\bbT(\pi)$ by $\la m \mid n \ra_\pi := \delta_{\lambda,\mu} \la m \mid n \ra_\lambda$ for all $\lambda,\mu \in \pi$, $m \in x_\lambda \Hc$, $n \in x_\mu \Hc$.

\begin{cor}
The two bases $\{ {}^\lambda C_w \mid \lambda \in \pi,\ w \in {}^\lambda W \}$ and $\{ {}^\lambda D_{w_\lambda w w_0} \mid \lambda \in \pi,\ w \in {}^\lambda W \}$ of $\bbT(\pi)$ are dual to each other with respect to the bilinear form $\la \cdot \mid \cdot \ra_\pi$. Moreover, for all $m,n \in \bbT(\pi)$ and $x \in \bbS(\pi)$, we have $\la m \mid xn \ra_\pi = \la x^\flat m \mid n \ra_\pi$.
\end{cor}

\subsection{Cell representations}
Let $X \in L(W)$ and $x \in X$. Set 
\begin{align}
\begin{split}
&C_{\leq_L X}(\pi) := \bigoplus_{\lambda \in \pi} \bigoplus_{\substack{w \in {}^\lambda W \\ w_\lambda w \leq_L x}} \AZ {}^\lambda C_w,\qqu D_{\leq_L}(\pi) := \bigoplus_{\lambda \in \pi} \bigoplus_{\substack{w \in {}^\lambda W \\ w \leq_L x}} \AZ {}^\lambda D_w, \\
&C_{<_L X}(\pi) := \bigoplus_{\lambda \in \pi} \bigoplus_{\substack{w \in {}^\lambda W \\ w_\lambda w <_L x}} \AZ {}^\lambda C_w,\qqu D_{\leq_L}(\pi) := \bigoplus_{\lambda \in \pi} \bigoplus_{\substack{w \in {}^\lambda W \\ w <_L x}} \AZ {}^\lambda D_w, \\
&C^L_X(\pi) := C_{\leq_L X}(\pi)/C_{<_L X}(\pi),\qu\qqu D^L_X(\pi) := D_{\leq_L X}(\pi) / D_{<_L X}(\pi).
\end{split} \nonumber
\end{align}
Note that these objects are independent of the choice of $x \in X$. We denote the image of $m \in C_{\leq_L X}(\pi)$ (resp., $D_{\leq_L X}(\pi)$) under the quotient map $C_{\leq_L X}(\pi) \rightarrow C^L_X(\pi)$ (resp., $D_{\leq_L X}(\pi) \rightarrow D^L_X(\pi)$) by $[m]_X$ (resp., $[m]'_X$).

\begin{prop}
Let $X \in L(W)$.
\begin{enumerate}
\item $C_{\leq_L X}(\pi)$ is a $\bbS(\pi)$-submodule of $\bbT(\pi)$.
\item $C_{<_L X}(\pi)$ is a $\bbS(\pi)$-submodule of $\bbT(\pi)$.
\item $C^L_{X}(\pi)$ is a left $\bbS(\pi)$-module having a basis $\{ [{}^\lambda C_w]_X \mid \lambda \in \pi,\ w \in {}^\lambda W \cap w_\lambda X \}$. Here, $w_\lambda X := \{ w_\lambda x \mid x \in X \}$.
\end{enumerate}
\end{prop}

\begin{proof}
We will prove only $(1)$ since the proof of $(2)$ is similar to that of $(1)$, and $(3)$ follows from $(1)$ and $(2)$. Fix $x \in X$. In order to show that $C_{\leq_L X}(\pi)$ is a $\bbS(\pi)$-submodule, it suffices to verify that $\xi_{\lambda,y,\mu} {}^\mu C_w \in C_{\leq_L X}(\pi)$ for all $\lambda,\mu \in \pi$, $y \in {}^\lambda W^\mu$, and $w \in {}^\mu W$ such that $w_\mu w \leq_L x$. By Proposition \ref{JCw} and Lemma \ref{property of eta} (2), we have
\begin{align}
\xi_{\lambda,y,\mu} {}^\mu C_w = \xi_{\lambda,y,\mu} C_{w_\mu w} = \frac{1}{P_\mu} \eta_{\lambda,y,\mu} C_{w_\mu w}. \nonumber
\end{align}
Also, by Lemma \ref{property of eta} (2), we have $\eta_{\lambda,y,\mu} = \xi_{\lambda,y,\mu}(x_\mu) \in x_\lambda \Hc$; one can write $\eta_{\lambda,y,\mu} = x_\lambda H$ for some $H \in \Hc$. Then, $H C_{w_\mu w}$ is a linear combination of $C_{w'}$, $w' \leq_L w_\mu w (\leq_L x)$. Hence, by Proposition \ref{from KL to parabolic KL}, $x_\lambda H C_{w_\mu w}$ is a linear combination of ${}^\lambda C_{w''}$ for $w'' \in {}^\lambda W$ with $w_\lambda w'' \leq_L w' (\leq_L x)$. Therefore, we have $\xi_{\lambda,y,\mu} {}^\mu C_w = \frac{1}{P_\mu} \eta_{\lambda,y,\mu} C_{w_\mu w} \in \frac{1}{P_\mu} C_{\leq_L X}(\pi)$. However, since $\xi_{\lambda,y,\mu} {}^\mu C_w \in x_\lambda \Hc = \bigoplus_{z \in {}^\lambda W} \AZ {}^\lambda C_z$, we conclude that $\xi_{\lambda,y,\mu} {}^\mu C_w \in C_{\leq_L X}(\pi)$. This completes the proof.
\end{proof}

Similarly, one can prove the following: $D_{\leq_L X}(\pi)$ and $D_{<_L X}(\pi)$ are $\bbS(\pi)$-submodules, and $D^L_X(\pi)$ is a left $\bbS(\pi)$-module having a basis $\{ [{}^\lambda D_w]'_X \mid \lambda \in \pi,\ w \in {}^\lambda W \cap X \}$.

%\begin{rem}\normalfont
%Let $X \in L(W)$. In general, we have $C_{\leq_L X}(\pi) \supset \bigoplus_{\lambda \in \pi} x_\lambda C_{\leq_L X}$. These coincide if we extend the base ring $\AZ$ to its quotient field $\Q(p,q)$. However, we always have $D_{\leq_L X}(\pi) = \bigoplus_{\lambda \in \pi} x_\lambda D_{\leq_L X}$.
%\end{rem}

\section{Global $\jmath$-crystal bases for the irreducible $\Uj$-modules}\label{existence of global jcry bases}
\subsection{Surjection $\xi:\Uj \rightarrow \bfS(\pij)$}
Let $\pij = \{ \lambda = (\lambda_0,\ldots,\lambda_r) \in \Z_{\geq 0}^{r+1} \mid \sum_{i=0}^r \lambda_i = d \}$. For $\lambda \in \pij$, set $I_\lambda = \{ 0,1,\ldots,d-1 \} \setminus \{ \lambda_0,\lambda_{0,1},\ldots,\lambda_{0,r-1} \}$, where $\lambda_{0,k} = \sum_{i=0}^k \lambda_i$.

Let $\bfV = \bigoplus_{i = -r}^r \Q(p,q) v_i$ be the vector representation of $\U$ with $v_{-r}$ a highest weight vector. Then, $\bfV^{\otimes d}$ has a basis $\{ v_{i_1,\ldots,i_d} := v_{i_1} \otimes \cdots \otimes v_{i_d} \mid -r \leq i_1,\ldots,i_d \leq r \}$. $\bfH := \Q(p,q) \otimes_{\AZ} \clH$ acts on $\bfV^{\otimes d}$ by
\begin{align}
\begin{split}
v_{i_1,\ldots,i_d} H_0 &= \begin{cases}
v_{-i_1,i_2,\ldots,i_d} \qu & \IF i_1 > 0, \\
p\inv v_{i_1,\ldots,i_d} \qu & \IF i_1 = 0, \\
v_{-i_1,i_2,\ldots,i_d} + (p\inv - p)v_{i_1,\ldots,i_d} \qu & \IF i_1 < 0,
\end{cases} \\
v_{i_1,\ldots,i_d} H_j &= \begin{cases}
v_{\ldots,i_{j+1},i_j,\ldots} \qu & \IF i_j < i_{j+1}, \\
q\inv v_{i_1,\ldots,i_d} \qu & \IF i_j = i_{j+1}, \\
v_{\ldots,i_{j+1},i_j,\ldots} + (q\inv - q)v_{i_1,\ldots,i_d} \qu & \IF i_j > i_{j+1}.
\end{cases}
\end{split} \nonumber
\end{align}
Then, it is easily seen that $\bfV^{\otimes d}$ is isomorphic to $\bfT(\pij) := \Q(p,q) \otimes_{\AZ} \bbT(\pij)$ as a right $\bfH$-module. Setting $\bfS(\pij) := \Q(p,q) \otimes_{\AZ} \bbS(\pij)$, $\bfV^{\otimes d}$ becomes a left $\bfS(\pij)$-module. By the double centralizer property between $\Uj$ and $\bfH$ on $\bfV^{\otimes d}$ (\cite{BW13}, \cite{BWW16}), there exists a surjective algebra homomorphism $\xi:\Uj \rightarrow \bfS(\pij)$. In particular, every $\bfS(\pij)$-modules are regarded as $\Uj$-modules via $\xi$. In \cite{W17}, it is proved that for each $\bflm \in \Pj$, the irreducible highest weight module $L(\bflm)$ is isomorphic to $\bfC^L_X(\pij) := \Q(p,q) \otimes_{\AZ} C^L_X(\pij)$ for some $X \in L(W_d)$, where $d = |\bflm|$.

For $i \in \Ij$, we define two maps $\etil_i,\ftil_i: \pij \rightarrow \pij \sqcup \{ 0 \}$, where $0$ denotes a formal symbol, as follows. Let $\lm = (\lm_0,\ldots,\lm_r) \in \pij$. Then, we set
$$
\etil_i \lm = \begin{cases}
(\lm_0,\ldots,\lm_{i-2},\lm_{i-1}+1,\lm_i-1,\lm_{i+1},\ldots,\lm_r) \qu & \IF \lm_i > 0, \\
0 \qu & \IF \lm_i = 0,
\end{cases}
$$
and
$$
\ftil_i \lm = \begin{cases}
(\lm_0,\ldots,\lm_{i-2},\lm_{i-1}-1,\lm_i+1,\lm_{i+1},\ldots,\lm_r) \qu & \IF \lm_{i-1} > 0, \\
0 \qu & \IF \lm_{i-1} = 0.
\end{cases}
$$
By convention, we set $\xi_{\lm,x,\mu} = 0$ if $\lm = 0$ or $\mu = 0$.

\begin{prop}
For $i \in \Ij$, we have
\begin{align}
\begin{split}
&\xi(e_i) = \sum_{\lambda \in \pij} \xi_{\etil_i(\lambda),e,\lambda}, \\
&\xi(f_i) = \sum_{\lambda \in \pij} \xi_{\ftil_i(\lambda),e,\lambda}.
\end{split} \nonumber
\end{align}
\end{prop}

\begin{proof}
We prove only the statement for $f_1$; the proofs for $f_i$, $i \neq 1$ and for $e_i$ are similar.
Recall the comultiplication $\Delta$ of $\bfU$; we have
\begin{align}
\Delta^{(d-1)}(E_i) = \sum_{k = 1}^d 1^{\otimes k-1} \otimes E_i \otimes (K_i\inv)^{\otimes d-k},\ \Delta^{(d-1)}(F_i) = \sum_{k = 1}^d K_i^{\otimes d-k} \otimes F_i \otimes 1^{\otimes k-1}. \nonumber
\end{align}
Then, we compute as
\begin{align}
\begin{split}
f_1v_\lambda &= pq\inv q^{\lambda_0} \sum_{k=1}^{\lambda_0} q^{\lambda_0-k} v_{0^{\lambda_0-k},1,0^{k-1},1^{\lambda_1},\ldots,r^{\lambda_r}} \\
&\qu + \sum_{k=1}^{\lambda_0} q^{\lambda_0-k} v_{0^{k-1},-1,0^{\lambda_0-k},1^{\lambda_1},\ldots,r^{\lambda_r}} \\
&= pq^{\lambda_0-1}\sum_{k=1}^{\lambda_0} q^{\lambda_0-k} v_{\ftil_1(\lambda)} H_{\lambda_0-1} \cdots H_{\lambda_0-(k-1)} \\
&\qu + \sum_{k=1}^{\lambda_0} q^{\lambda_0-k} v_{\ftil_1(\lambda)} H_{\lambda_0-1} \cdots H_1 H_0 H_1 \cdots H_{k-1} \\
&= pq^{2(\lambda_0-1)} \sum_{k=1}^{\lambda_0} q^{-k+1} v_{\ftil_1(\lambda)} H_{\lambda_0-1} \cdots H_{\lambda_0-(k-1)} \\
&\qu + pq^{2(\lambda_0-1)}\sum p\inv q^{-(\lambda_0+k-2)} v_{\ftil_1(\lambda)} H_{\lambda_0-1} \cdots H_1 H_0 H_1 \cdots H_{k-1} \\
&= \xi_{\ftil_1(\lambda),e,\lambda}(v_{\lambda}).
\end{split} \nonumber
\end{align}
This proves the assertion.
\end{proof}

Here are immediate consequences.

\begin{cor}
Let $x \in \Uj$. Then, $\xi(\sigmaj(x)) = \xi(x)^\flat$, $\xi(\psij(x)) = \ol{\xi(x)}$.
\end{cor}

\begin{cor}
The bilinear form $\la \cdot \mid \cdot \ra_{\pij}$ of $\bfT(\pij)$ satisfies
\begin{align}
\la xm \mid n \ra_{\pij} = \la m \mid \sigmaj(x)n \ra_{\pij} \nonumber
\end{align}
for all $x \in \Uj$, $m,n \in \bfT(\pij)$.
\end{cor}

\subsection{Global $\jmath$-crystal basis of irreducible $\Uj$-module}
Let $X \in L(W)$. Then, $\bfC^L_X(\pij) \simeq L(\bflm)$ for some $\bflm \in \Pj$. Since $L(\bflm)$ is a highest weight module, there exists a unique $\lambda \in \pij$ and $w \in {}^\lambda W$ such that $[{}^\lambda C_w]_X \in \bfC^L_X(\pij)$ is a highest weight vector.

Recall the isomorphism $D^L_{Xw_0} \simeq C^L_X$ of left $\clH$-modules. Set $\bfC^L_X := \Q(p,q) \otimes_{\AZ} C^L_X$, and define $\bfD^L_{Xw_0}$ and $\bfD^L_{Xw_0}(\pij)$ similarly. Then, we have
$$
\bfD^L_{Xw_0}(\pij) \simeq \bfT(\pij) \otimes_{\bfH} \bfD^L_{Xw_0} \simeq \bfT(\pij) \otimes_{\bfH} \bfC^L_X \simeq \bfC^L_X(\pij)
$$
as left $\Uj$-modules. Hence, $[{}^\lambda D_{w_\lambda w w_0}]'_{Xw_0} \in \bfD^L_{Xw_0}(\pij)$ is also a highest weight vector. Thus, we obtain two isomorphisms
\begin{align}
\begin{split}
&\vphi_C: L(\bflm) \rightarrow \bfC^L_X(\pij);\ v_\bflm \mapsto [{}^\lm C_w]_X, \\
&\vphi_D: L(\bflm) \rightarrow \bfD^L_{Xw_0}(\pij);\ v_\bflm \mapsto [{}^\lm D_{w_\lm ww_0}]'_{Xw_0}
\end{split} \nonumber
\end{align}
of $\Uj$-modules, where $v_\bflm \in L(\bflm)$ is a fixed highest weight vector.

\begin{defi}\normalfont
Let $\bflm \in \Pj$ and $v_\bflm \in L(\bflm)$ be a highest weight vector. Define the bilinear form $(\cdot,\cdot)_1$ on $L(\bflm)$ by $(v_\bflm,v_\bflm)_1 = 1$ and $(xm,n)_1 = (n,\sigmaj(x)n)_1$ for all $x \in \Uj$, $m,n \in L(\bflm)$.
\end{defi}

\begin{prop}\label{(,)_1 is nondegenerate}
Let $\bflm \in \Pj$. Then, the bilinear form $(\cdot,\cdot)_1$ is nondegenerate.
\end{prop}

\begin{proof}
For $m,n \in L(\bflm)$, set $(m,n) := \la \vphi_C(m) \mid \vphi_D(n) \ra_{\pij}$. Then, we have
$$
(v_\bflm,v_\bflm) = \la [{}^\lm C_w]_X \mid [{}^\lm D_{w_\lm ww_0}]'_{Xw_0} \ra_{\pij} = 1,
$$
and
$$
(xm,n) = \la x \vphi_C(m) \mid \vphi_D(m) \ra_{\pij} = \la \vphi_C(m) \mid \sigmaj(x) \vphi_D(n) \ra_{\pij} = (m,\sigmaj(x) n).
$$
Hence, we have $(\cdot,\cdot) = (\cdot,\cdot)_1$. Then, it is clear that $\{ \vphi_C\inv([{}^\mu C_y]_X) \mid \mu \in \pij,\ y \in {}^\mu W \cap w_\mu X \}$ and $\{ \vphi_D\inv([{}^\mu D_{w_\mu yw_0}]'_{Xw_0}) \mid \mu \in \pij,\ y \in {}^\mu W \cap Xw_0 \}$ form bases which are dual to each other with respect to $(\cdot,\cdot)_1$. This proves the proposition.
\end{proof}

Recall that the set $\{ (\mu,y) \mid \mu \in \pij,\ y \in {}^\mu W \cap w_\mu X \}$ is identical to $\clB(\bflm)$. For each $b \in \clB(\bflm)$, set
$$
\Gjlow(b) := \vphi_C\inv([{}^\mu C_y]_X), \qu \Gjup(b) := \vphi_D\inv([{}^\mu D_{w_\mu yw_0}]'_{Xw_0}),
$$
where $(\mu,y)$ is the pair corresponding to $b$. Then, $\Gjlow(\bflm) := \{ \Gjlow(b) \mid b \in \clB(\bflm) \}$ and $\Gjup(\bflm) := \{ \Gjup(b) \mid b \in \clB(\bflm) \}$ are bases of $L(\bflm)$.

\begin{defi}\normalfont
Let $\bflm \in \Pj(d)$, and $v_{\bflm} \in L(\bflm)$ be a highest weight vector. Define a bilinear form $(\cdot,\cdot)_2$ on $L(\bflm)$, and a $\psij$-involution $\psij_{\bflm}$ on $L(\bflm)$ by
\begin{align}
\begin{split}
&(v_{\bflm},v_{\bflm})_2 = 1,\ (xm,n)_2 = (m,\tauj(x)n)_2 \qu \Forall x \in \Uj,\ m,n \in L(\bflm), \\
&\psij_{\bflm}(v_{\bflm}) = v_{\bflm}.
\end{split} \nonumber
\end{align}
\end{defi}

Let $(\clL(\bflm),\clB(\bflm))$ be the unique $\jmath$-crystal basis of $L(\bflm)$ such that $v_\bflm + q\clL(\bflm) \in \clB(\bflm)$.

\begin{theo}
Let $\bflm \in \Pj(d)$. Then, the following hold.
\begin{enumerate}
\item $\psij_\bflm(\Gjlow(b)) = \Gjlow(b)$ for all $b \in \clB(\bflm)$.
\item $\psij_\bflm(\Gjup(b)) = \Gjup(b)$ for all $b \in \clB(\bflm)$.
\item $\Gjlow(\bflm)$ and $\Gjup(\bflm)$ are dual bases with respect to $(\cdot,\cdot)_1$.
\item $\clL(\bflm) = \{ m \in L(\bflm) \mid (m,m)_2 \in \Ao \}$. Consequently, $(\cdot,\cdot)_2$ induces the bilinear form $(\cdot,\cdot)_0$ on $\clL(\bflm)/q\clL(\bflm)$ defined by $(m+q\clL(\bflm),n+q\clL(\bflm))_0 := \lim_{q \rightarrow 0}(\lim_{p \rightarrow 0}(m,n)_2)$.
\item $\{ \Gjlow(b) \mid b \in \clB(\bflm) \}$ forms an almost orthonormal basis with respect to $(\cdot,\cdot)_2$, i.e., we have $(\Gjlow(b),\Gjlow(b'))_2 \in \delta_{b,b'} + q\Ao$ for all $b,b' \in \clB(\bflm)$.
\item $(b,b')_0 = \delta_{b,b'}$ for all $b,b' \in \clB(\bflm)$.
%\item $\Gjlow(b)$ is uniquely characterized by the conditions $(1)$ and $(\Gjlow(b),\Gjlow(b))_2 \in 1 + q\Ao$.
\item Let $L(\bflm)_\A$ be the $\A$-span of $\Gjlow(\bflm)$. Then, $(\clL(\bflm),L(\bflm)_\A,\psij_\bflm(\clL(\bflm)))$ is balanced. Moreover, the global basis associated to $\clB(\bflm)$ is $\{ \Gjlow(b) \mid b \in \clB(\bflm) \}$. In particular, $L(\bflm)$ has a global $\jmath$-crystal basis.
\end{enumerate}
\end{theo}

\begin{proof}
Items $(1)$ and $(2)$ are obvious from the definition of $\Gjlow(b)$ and $\Gjup(b)$. Item $(3)$ follows from the proof of Proposition \ref{(,)_1 is nondegenerate}.

To prove the rest, observe that $L(\bflm)$ is realized as a subquotient of $\bfV^{\otimes d}$ by using Kazhdan-Lusztig basis elements. To be precise, let $X \in \Pj$ be such that $L(\bflm) \simeq \bfC^L_X(\pij)$ and $x \in X$. Then,
$$
\bfC^L_X(\pij) = \frac{\Span_{\Q(p,q)}\{ {}^\lm C_w \mid \lm \in \pij,\ w \in {}^\lm W,\ w_\lm w \leq_L x \}}{\Span_{\Q(p,q)}\{ {}^\mu C_y \mid \mu \in \pij,\ y \in {}^\mu W,\ w_\mu y <_L x \}}.
$$
Then, items $(4)$-$(6)$ follows from the proof of \cite[Proposition 7.4.4]{W17}. To prove item $(7)$, it suffices to show that $L(\bflm)_\A$ is a $\UjA$-module. It follows from the fact that the $\A$-submodule of $\bfV^{\otimes d}$ spanned by the Kazhdan-Lusztig basis is a $\UA$-module, and that $\UjA \subset \UA$.
\end{proof}

\section{Basic properties of global crystal bases}\label{basic properties}
\subsection{Global crystal bases}
In this subsection, we exposite some basic properties concerning global crystal bases of $\U$-modules in $\Oint$. Let $M \in \Oint$, $(\clL,\clB)$ a crystal basis of $M$, $\psi_M$ a $\psi$-involution, and $M_\A$ a $\UA$-submodule of $M$. Suppose that $M$ has a global basis $G(\clB)$ with the associated balanced triple $(\clL,M_\A,\psi_M(\clL))$.

\begin{prop}[{\cite{K93}}]\label{fund prop for global crystal basis}
Let $i \in \bbI$, $b \in B$ and $m \in \Z_{\geq 0}$. Then, we have the following.
\begin{enumerate}
\item $\sum_{n \geq m} F_i^{(n)}M_{\A} = \bigoplus_{\substack{b' \in \clB \\ \vep_i(b') \geq m}} \A \Gj(b')$.
\item $\sum_{n \geq m} E_i^{(n)}M_{\A} = \bigoplus_{\substack{b' \in \clB \\ \vphi_i(b') \geq m}} \A \Gj(b')$.
\item $F_i \Gj(b) = [\vep_i(b)+1]\Gj(\Ftil_i b) + \sum_{\substack{b' \in \clB \\ \vep_i(b') > \vep_i(b)+1}} \vphi^{(i)}_{b',b} \Gj(b')$ for some $\vphi^{(i)}_{b',b} \in q^{2-\vep_i(b')}\Q[q]$.
\item $E_i \Gj(b) = [\vphi_i(b)+1]\Gj(\Etil_i b) + \sum_{\substack{b' \in \clB \\ \vphi_i(b') > \vphi_i(b)+1}} \vep^{(i)}_{b',b} \Gj(b')$ for some $\vep^{(i)}_{b',b} \in q^{2-\vphi_i(b')}\Q[q]$.
\end{enumerate}
\end{prop}

For $\lm \in P(M)$, set $I_\lm(M)$ to be the sum of submodules of $M$ isomorphic to $L(\lm)$. Also, we set
\begin{align}
\begin{split}
&W_{\succeq \lm}(M) := \sum_{\mu \succeq \lm} I_\mu(M), \\
&W_{\succ \lm}(M) := \sum_{\mu \succ \lm} I_\mu(M), \\
&W_\lm(M) := W_{\succeq \lm}(M)/W_{\succ \lm}(M).
\end{split} \nonumber
\end{align}

\begin{theo}[{\cite{K93}, \cite{L94}}]
Let $M,\clL,\clB,M_\A$ be as above. Then, for each $\lm \in P(M)$, the following hold:
\begin{enumerate}
\item $W_{\succeq \lm}(M)$ has a global crystal basis $W_{\succeq \lm}(G(\clB)) := \{ G(b) \mid I(b) \succeq \lm \}$ with the associated balanced triple $(W_{\succeq \lm}(\clL),W_{\succeq \lm}(M_{\A}),W_{\succeq \lm}(\psi_M(\clL)))$, where $W_{\succeq \lm}(\clL) := W_{\succeq \lm}(M) \cap \clL$, and so on.
\item $W_{\succ \lm}(M)$ has a global crystal basis $W_{\succ \lm}(G(\clB)) := \{ G(b) \mid I(b) \succ \lm \}$ with the associated balanced triple $(W_{\succ \lm}(\clL),W_{\succ \lm}(M_{\A}),W_{\succ \lm}(\psi_M(\clL)))$, where $W_{\succ \lm}(\clL) := W_{\succ \lm}(M) \cap \clL$, and so on.
\item $W_{\lm}(M)$ has a global crystal basis $W_{\lm}(G(\clB)) := \{ G(b) + W_{\succ \lm}(M) \mid I(b) = \lm \}$ with the associated balanced triple $(W_{\lm}(\clL),W_{\lm}(M_{\A}),W_{\lm}(\psi_M(\clL)))$, where $W_{\lm}(\clL) := W_{\succeq \lm}(\clL)/W_{\succ \lm}(\clL)$, and so on.
\item There exists a $\U$-module isomorphism $\xi: L(\lm)^{\oplus m_\lm} \rightarrow W_\lm(M)$ which induces an isomorphism
$$
(\clL(\lm)^{\oplus m_\lm}, (L(\lm)_{\A})^{\oplus m_\lm}, \psi_\lm(\clL(\lm))^{\oplus m_\lm}) \simeq (W_{\lm}(\clL),W_{\lm}(M_{\A}),W_{\lm}(\psi_M(\clL))),
$$
where $m_\lm := \dim \Hom_{\U}(L(\lm),M)$ denotes the multiplicity of $L(\lm)$ in $M$.
\end{enumerate}
\end{theo}

\begin{rem}\label{filtration for frp}\normalfont
By replacing $P(M)$ with $\Pj(M)$ and $\preceq$ with $\trilefteq$, the same result holds for integrable modules over $\bfU(\frl)$ with global crystal bases.
\end{rem}

\subsection{$\jmath$-canonical bases}
Let $M \in \Oint$ be a based $\U$-module with a crystal basis $(\clL,\clB)$, a global crystal basis $G(\clB)$, a $\psi$-involution $\psi_M$, and a balanced triple $(\clL,M_{\A},\psi_M(\clL))$. Set $\psij_M := \Upsilon \circ \psi_M$. We denote by $\Gj(\clB)$ the associated $\jmath$-canonical basis. Recall that $\psij_M$ is a $\psij$-involution on $M$, and $(\clL,M_\A,\psij_M(\clL))$ is a balanced triple with the associated global basis $\Gj(\clB)$.

\begin{lem}\label{expansion of jCB in CB}
Let $b \in \clB$. Let us write as
$$
\Gj(b) = G(b) + \sum_{\substack{b' \in \clB \\ \wtj(b') = \wtj(b) \AND \wt(b') < \wt(b)}} c_{b',b} G(b')
$$
for some $c_{b',b} \in q\Ao \cap \A$. Then, we have $c_{b',b} = 0$ unless
\begin{align}\label{jcanonical condition}
I^\jmath(b) \trilefteq I^\jmath(b') \OR |I^\jmath(b')^-| < |I^\jmath(b)^-|.
\end{align}
\end{lem}

\begin{proof}
By the construction of $\Gj(b)$, it suffices to show that $\psij_M(G(b))$ is a linear combination of $G(b')$ with $b'$ satisfying \eqref{jcanonical condition}. Since $\psij_M(G(b)) = \Upsilon G(b) \in \bfU^- G(b)$, it suffices to show that for each $l \in \Z_{\geq 0}$ and $i_1,\ldots,i_l \in \bbI$, we have
$$
F_{i_l} \cdots F_{i_1} G(b) \in \Span_{\Q(p,q)} \{ G(b') \mid \text{$b'$ satisfies condition \eqref{jcanonical condition}} \}.
$$
We prove it by induction on $l$. When $l = 0$, there are nothing to prove. So, assume that $l > 0$ and that $F_{i_{l-1}} \cdots F_{i_1} G(b) \in \Span_{\Q(p,q)} \{ G(b') \mid \text{$b'$ satisfies condition \eqref{jcanonical condition}} \}$ for all $i_1,\ldots,i_{l-1} \in \bbI$. If $i_l \neq \ul{1}$, then, by Remark \ref{filtration for frp}, we have
$$
F_{i_l} G(b') \in \Span_{\Q(p,q)} \{ G(b'') \mid I^\jmath(b') \trilefteq I^\jmath(b'') \}
$$
for all $b'$ satisfying condition \eqref{jcanonical condition}. Since $|I^\jmath(b'')^-| = |I^\jmath(b')^-|$ for all $b''$ with $I^\jmath(b') \trilefteq I^\jmath(b'')$, $b''$ satisfies condition \eqref{jcanonical condition}.

If $i_l = \ul{1}$, then $\wt(F_{i_l} G(b')) = \wt(G(b')) - \alpha_{\ul{1}}$. This immediately implies that $F_{i_l} G(b') \in \Span_{\Q(p,q)} \{ G(b'') \mid |I^\jmath(b'')^-| < |I^\jmath(b')^-| \}$. Therefore, $F_{i_l} \cdots F_{i_1} G(b)$ is a linear combination of $G(b')$ with $|I^\jmath(b')^-| < |I^\jmath(b)^-|$. Thus, the proof completes.
\end{proof}

\begin{prop}\label{str const for Gj-basis}
Let $b \in \clB$ and $i \in \Ij \setminus \{1\}$. Then, we have
\begin{align}
\begin{split}
e_i \Gj(b) &= [\vphi_{\ul{i}}(b) + 1] \Gj(\Etil_{\ul{i}}b) + \sum_{\substack{b' \in \clB \setminus \{ \Etil_{\ul{i}}b \} \\ \wtj(b') = \wtj(b) + \gamma_i \AND \wt(b') \leq \wt(b) + \alpha_{\ul{i}}}} e_{b',b}^{(i)} \Gj(b'), \\
f_i \Gj(b) &= [\vphi_{-\ul{i}}(b) + 1] \Gj(\Etil_{-\ul{i}}b) + \sum_{\substack{b' \in \clB \setminus \{ \Etil_{-\ul{i}}b \} \\ \wtj(b') = \wtj(b) - \gamma_i \AND \wt(b') \leq \wt(b) + \alpha_{-\ul{i}}}} f_{b',b}^{(i)} \Gj(b')
\end{split} \nonumber
\end{align}
for some $e_{b',b}^{(i)}, f_{b',b}^{(i)} \in \A$. Moreover, $e_{b',b}^{(i)} = f_{b',b}^{(i)} = 0$ unless $I^\jmath(b) \trilefteq I^\jmath(b')$ or $|I^\jmath(b')^-| < |I^\jmath(b)^-|$.
\end{prop}

\begin{proof}
We prove the assertion only for $e_i$; the proof for $f_i$ is similar. By Lemma \ref{expansion of jCB in CB}, we can write
$$
\Gj(b) = G(b) + \sum_{b' \in \clB \setminus \{ b \}} c_{b',b} G(b')
$$
for some $c_{b',b} \in \A$ such that $c_{b',b} = 0$ unless $I^\jmath(b) \trilefteq I^\jmath(b')$ or $|I^\jmath(b')^-| < |I^\jmath(b)^-|$. Since $e_i \in U_q(\frl)$, it holds that
$$
e_i \Gj(b) \in \Span_\A \{ G(b'') \mid I^\jmath(b) \trilefteq I^\jmath(b') \OR |I^\jmath(b')^-| < |I^\jmath(b)^-| \}.
$$
Hence, it suffices to show that $[e_i \Gj(b) : \Gj(\Etil_{\ul{i}}b)] = [\vphi_{\ul{i}}(b) + 1]$. By the definitions of $e_i$ and $\Gj(b)$, $e_i \Gj(b)$ is the sum of $E_{\ul{i}} G(b)$ and a linear combination of weight vectors of $M$ of weight lower than $\wt(b) + \alpha_{\ul{i}}$. We know from Proposition \ref{fund prop for global crystal basis} $(4)$ that $[E_{\ul{i}} \Gj(b) : G(\Etil_{\ul{i}}b)] = [\vphi_{\ul{i}}(b)+1]$. Hence, we have $[e_i \Gj(b) : \Gj(\Etil_{\ul{i}}b)] = [\vphi_{\ul{i}}(b) + 1]$. This proves the assertion.
\end{proof}

\subsection{Global $\jmath$-crystal bases}\label{subsec: global basis}
Let $M \in \Oj$, $(\clL,\clB)$ a $\jmath$-crystal basis of $M$, $\psij_M$ a $\psij$-involution, and $M_{\A}$ a $\UjA$-submodule of $M$. Suppose that $M$ has a $\jmath$-global basis $\Gj(\clB)$ with the associated balanced triple $(\clL,M_{\A},\psij_M(\clL))$.

%\begin{lem}
%Let $i \in \Ij$, $b \in B$.
%\begin{enumerate}
%\item For each $m \in \Z_{\geq 0}$, the subspace $\sum_{n \geq m} f_i^{(n)} M_{\A}$ is the free $\A$-lattice with a free basis $\{ \Gj(b') \mid \vep_i(b') \geq m \}$.
%%
%\item We have $f_i \Gj(b) = [\vep_i(b) + 1] \Gj(\etil_i b) + \sum_{\substack{b' \in B \\ \vep_i(b') > \vep_i(b) + 1}} \vphi_{b',b} \Gj(b')$ for some $\vphi_{b',b} \in q^{2-\vep_i(b')} \A_0 \cap \A$ satsifying $\ol{\vphi_{b',b}} = \vphi_{b',b}$.
%%
%\item Suppose that $i \neq 1$. Then, we have $e_i \Gj(b) = [\vphi_i(b) + 1]\Gj(\etil_i b) + \sum_{\substack{b' \in B \\ \vphi_i(b') > \vphi_i(b) + 1}} \vep_{b',b} \Gj(b')$ for some $\vep_{b',b} \in q^{2-\vphi_i(b')} \A_0 \cap \A$ satisfying $\ol{\vep_{b',b}} = \vep_{b',b}$.
%\end{enumerate}
%\end{lem}

Following \cite{K02}, let us introduce modified Kashiwara operators:
\begin{defi}
For $n \in \Z$, set
\begin{align}
\begin{split}
\ftil_i^{(n)} &:= \sum_{t \geq 0,-n} f_i^{(n+t)} e_i^{(t)} A_n(t;k_i), \\
\ftil_1^{(n)} &:= \sum_{t \geq 0,-n} f_1^{(n+t)} e_1^{(t)} a_n(t;k_1),
\end{split} \nonumber
\end{align}
where
\begin{align}
\begin{split}
A_n(t;x) &:= (-1)^t q^{t(1-n)} x^t \prod_{s=0}^{t-1}(1-q^{n+2s}), \\
a_n(t;x) &:= (-1)^t p^t q^{t(1-n)} x^t \prod_{s=0}^{t-1} q^s(1-q^{n+2s}).
\end{split} \nonumber
\end{align}
\end{defi}

\begin{lem}\label{modified ftil preserves lattice}
Let $M \in \Oj$ with the $\jmath$-crystal basis $(\clL,\clB)$. For $n \in \Z$, we have $\ftil_i^{(n)} \clL \subset \clL$, and $\ftil_i^{(n)} \clL = \ftil_i^n \clL$ modulo $q\clL$.
\end{lem}

\begin{proof}
If $i \neq 1$, then the statement follows from \cite[Proposition 6.1]{K02}. Hence, we prove the case when $i=1$. It suffices to prove the following: For each $u \in \clL$ such that $e_1u = 0$, $k_1u = q^au$, $e_1f_1u = [b]\{ a-b-1 \}u$ with $a \in \Z$ and $b \in \Z_{\geq 0}$, we have $\ftil_1^{(n)} f_1^{(m)} u = cf_1^{(m+n)}$ for some $c \in 1 + q\A_0 \cap \A$. First of all, we have
\begin{align}
\ftil_1^{(n)} f_1^{(m)} u = \sum_{t \geq 0,-n} a_n(t;q^{a-3m}) {m+n \brack m-t} {b-m+t \brack t} \prod_{s=0}^{t-1} \{ a-b-m+s \} f_1^{(m+n)} u. \nonumber
\end{align}
We compute the coefficient, say $A$, of the right-hand side as follows.
\begin{align}
\begin{split}
A &= \sum_{t \geq 0,-n} A_n(t;q^{b-2m}) {m+n \brack m-t} {b-m+t \brack t} \prod_{s=0}^{t-1}(1 + p^2q^{2(a-b-m+s)}) \\
&= \sum_{t \geq 0,-n} B_t + p^2\sum_{t \geq 0,-n} B_t g_t,
\end{split} \nonumber
\end{align}
where $B_t := A_n(t;q^{b-2m}) {m+n \brack m-t} {b-m+t \brack t}$, and $g_t \in \Z[p,q,q\inv]$ with $\prod_{s=0}^{t-1}(1 + p^2q^{2(a-b-m+s)}) = 1 + p^2g_t$. By the proof of \cite[Proposition 6.1]{K02}, we have $B_t \in 1 + q\Z[q]$. Also, it is clear that $p^2\sum_{t \geq 0,-n} B_t g_t \in p^2\Z[p,q,q\inv]$. Thus, we have $\ftil_1^{(n)} f_1^{(m)} u \in \clL$ and $\ftil_1^{(n)} f_1^{(m)} u = f_1^{(m+n)} u = \ftil_1^n f_1^{(m)} u$ modulo $q\clL$. This proves the lemma.
\end{proof}

\begin{prop}
Let $i \in \Ij$, $b \in \clB$ and $m \in \Z_{\geq 0}$. Then, we have the following.
\begin{enumerate}
\item $\sum_{n \geq m} f_i^{(n)}M_{\A} = \bigoplus_{\substack{b' \in \clB \\ \vep_i(b') \geq m}} \A \Gj(b')$.
\item $\sum_{n \geq m} e_i^{(n)}M_{\A} = \bigoplus_{\substack{b' \in \clB \\ \vphi_i(b') \geq m}} \A \Gj(b')$ if $i \neq 1$.
\item $f_i \Gj(b) = [\vep_i(b)+1]\Gj(\ftil_i b) + \sum_{\substack{b' \in \clB \\ \vep_i(b') > \vep_i(b)+1}} \vphi^{(i)}_{b',b} \Gj(b')$ for some $\vphi^{(i)}_{b',b} \in q^{2-\vep_i(b')}\Q[q]$.
\item $e_i \Gj(b) = [\vphi_i(b)+1]\Gj(\etil_i b) + \sum_{\substack{b' \in \clB \\ \vphi_i(b') > \vphi_i(b)+1}} \vep^{(i)}_{b',b} \Gj(b')$ for some $\vep^{(i)}_{b',b} \in q^{2-\vphi_i(b')}\Q[q]$ if $i \neq 1$.
\end{enumerate}
\end{prop}

\begin{proof}
Since $(e_i,k_i,f_i)$, $i \neq 1$ forms an $\frsl_2$-triple, most of the assertions follows from Proposition \ref{fund prop for global crystal basis}. What we have to prove are assertions $(1)$ and $(3)$ for $i = 1$. First, we prove part $(1)$ by induction on $m$. When $m = 0$, the both sides of the equation to be proved are $0$. Assume that assertion $(1)$ holds for all $m' > m$. Let $b' \in \clB$ be such that $\vep_1(b') = m$. Set $b'_0 := \etil_1^mb$, and consider $u := \ftil_1^{(m)} \Gj(b'_0)$. By the definition of $\ftil_1^{(m)}$ and Lemma \ref{modified ftil preserves lattice}, we have
$$
u - f_1^{(m)} \Gj(b'_0) \in \sum_{n > m} f_1^{(n)} M_\A \AND u + q\clL = b'.
$$
By our inductive hypothesis, we can write
$$
u - f_1^{(m)}\Gj(b'_0) = \sum_{\substack{b'' \in \clB \\ \vep_1(b'') > m}} a_{b''} \Gj(b'')
$$
for some $a_{b''} \in \A$. Then, we can take $a'_{b''} \in q\Q[q]$ in a way such that $a_{b''} - \ol{a_{b''}} = a'_{b''} - \ol{a'_{b'''}}$. Set $v := u - \sum_{b''} a'_{b''} \Gj(b'') = f_1^{(m)}\Gj(b'_0) + \sum_{b''} (a_{b'''} - a'_{b''}) \Gj(b'')$. Then, we have $v \in M_\A \cap \clL$, $\psij_M(v) = v$, and $v + q\clL = u + q\clL = b'$. These implies that $v = \Gj(b')$, and therefore, $\Gj(b') \in \sum_{n \geq m} f_1^{(n)} M_\A$. Hence, we obtain $\sum_{n \geq m} f_i^{(n)}M_{\A} \supset \bigoplus_{\substack{b' \in \clB \\ \vep_i(b') \geq m}} \A \Gj(b')$.

We prove the opposite inclusion. For each $\lm \in \Lmj$, we have
\begin{align}
\begin{split}
(M_\A)_\lm &\subset \sum_{b \in \clB_\lm} \A \Gj(b) \\
&= \sum_{\substack{b \in \clB_\lm \\ \vep_1(b) = 0}} \A \Gj(b) + \sum_{\substack{b' \in \clB_\lm \\ \vep_1(b') \geq 1}} \A \Gj(b') \\
&\subset \sum_{\substack{b \in \clB_\lm \\ \vep_1(b) = 0}} \A \Gj(b) + \sum_{n \geq 1} f_1^{(n)} (M_\A)_{\lm + n\gamma_1}.
\end{split} \nonumber
\end{align}
Hence, we obtain
\begin{align}
\begin{split}
f_1^{(m)} (M_\A)_\lm & \subset \sum_{\substack{b \in \clB_\lm \\ \vep_1(b) = 0}} \A f_1^{(m)}\Gj(b) + \sum_{n \geq 1} f_1^{(m)}f_1^{(n)} (M_\A)_{\lm + n\gamma_1} \\
&\subset \sum_{\substack{b \in \clB_\lm \\ \vep_1(b) = 0}} \A f_1^{(m)}\Gj(b) + \sum_{n \geq 1} f_1^{(m+n)} (M_\A)_{\lm + n\gamma_1} \\
&= \sum_{\substack{b \in \clB_\lm \\ \vep_1(b) = 0}} \A f_1^{(m)}\Gj(b) + \sum_{\substack{b' \in \clB_\lm \\ \vep_1(b') > m}} \A \Gj(b') \qu (\text{by induction hypothesis}).
\end{split} \nonumber
\end{align}
Also, by the argument above, $f_1^{(m)} \Gj(b)$ with $\vep_1(b) = 0$ is contained in $\sum_{\vep(b') \geq m} \A \Gj(b')$. This completes the proof of part $(1)$.

Next, we turn to prove assertion $(3)$ for $i = 1$ by descending induction on $m := \vep_1(b)$. When $m$ is maximum among $\{ \vep_1(b') \mid b' \in \clB \}$, we have by $(1)$ that
$$
f_1 \Gj(b) \in \sum_{n > m} f_1^{(n)} M_\A = \sum_{\vep_1(b') > m} \A \Gj(b') = 0,
$$
and the equation in $(3)$ holds. Assume that $(3)$ is true for all $m' > m$. As in the proof of $(1)$, let us write
\begin{align}
\begin{split}
\Gj(b) &= f_1^{(m)} \Gj(\etil_1^m b) + \sum_{\substack{b' \in \clB \\ \vep_1(b') > m}} c_{b'} \Gj(b'), \\
\Gj(\ftil_1b) &= f_1^{(m+1)} \Gj(\etil_1^m b) + \sum_{\substack{b'' \in \clB \\ \vep_1(b'') > m+1}} d_{b''} \Gj(b'')
\end{split} \nonumber
\end{align}
for some $c_{b'},d_{b''} \in \A$. Then, we have
\begin{align}
\begin{split}
f_1 \Gj(b) &= [m+1] f_1^{(m+1)} \Gj(\etil_1^m b) + \sum_{\vep_1(b') > m} c_{b'} f_1 \Gj(b') \\
&=  [m+1] f_1^{(m+1)} \Gj(\etil_1^m b) + \sum_{\vep_1(b') > m} c_{b'} ([\vep_1(b')+1] \Gj(\ftil_1b') + \sum_{\vep_1(b'') > \vep_1(b')+1} \vphi_{b'',b'}^{(1)} \Gj(b'')) \\
&= [m+1] \Gj(\ftil_1b) + \sum_{\vep_1(b') > m} c_{b'} ([\vep_1(b')+1] \Gj(\ftil_1b') + \sum_{\vep_1(b'') > \vep_1(b')+1} \vphi_{b'',b'}^{(1)} \Gj(b'')) \\
&\qu \qu - \sum_{\substack{b'' \in \clB \\ \vep_1(b'') > m+1}} [m+1] d_{b''} \Gj(b'').
\end{split} \nonumber
\end{align}
Thus, we obtain that $f_i \Gj(b) = [\vep_i(b)+1]\Gj(\ftil_i b) + \sum_{\substack{b' \in \clB \\ \vep_i(b') > \vep_i(b)+1}} \vphi^{(i)}_{b',b} \Gj(b')$ for some $\vphi^{(i)}_{b',b} \in \A$. It remains to prove that $\vphi^{(i)}_{b',b} \in q^{2-\vep_1(b')} \Q[q]$. Let us write
$$
\Gj(b) = \sum_{k \geq m} f_1^{(k)} u_k
$$
for some $u_k \in \clL_{\wtj(b) + k\gamma_1}$ such that $e_1u_k = 0$. Note that $\Gj(b) + q\clL = u_m + q\clL$. Then, we have
$$
f_1 \Gj(b) = [m+1] f_1^{(m+1)}u_m + \sum_{k > m} [k+1] f_1^{(k+1)} u_k,
$$
and that $f_1^{(m+1)} u_m \in \clL$, $f_1^{(m+1)}u_m + q\clL = \ftil_1 b$. Hence, we have $f_1 \Gj(b) = [m+1] \Gj(\ftil_1b) + \sum_{k > m} [k+1] f_1^{(k+1)} u_k$ modulo $q^{2-m}\clL$. Then, rewriting $f_1^{(k+1)}u_k$ as a sum of $\Gj(b')$, $\vep_1(b') \leq k+1$ with coefficients in $q\Ao$, we conclude that the coefficient of $\Gj(b')$ in $f_1 \Gj(b)$ lies in $q^{2-\vep_1(b')}\Ao \cap \A = q^{2-\vep_1(b')} \Q[q]$. This completes the proof.
\end{proof}

For a bipartition $\bflm \in \Pj(M)$, define $I_\bflm(M)$, $W_{\succeq \bflm}(M)$, $W_{\succ \bflm}(M)$, and $W_\bflm(M)$ in a similar way as $I_\lm$, $W_{\succeq \lm}$, $W_{\succ \lm}$, and $W_\lm$, respectively.

\begin{defi}\normalfont
We say that $M$ has the property $(\ast)$ if there exists a poset $(S,\leq)$ and a map $s: \clB \rightarrow S$ satisfying the following:
\begin{enumerate}
\item The abelian group $Q := \sum_{i \in \bbI} \Z \alpha_i$ acts on $S$ freely; the action is written additively.
\item $\sigma \leq \sigma + \lm$ for all $\lm \in Q_+$, $\sigma \in S$.
\item $\sigma + \lm \leq \sigma' + \lm$ for all $\lm \in Q$, $\sigma \leq \sigma' \in S$.
\item $s(b) = s(b')$ only if $\wt(b) = \wt(b')$ for all $b,b' \in \clB$.
\item For $b \in \clB$ and $i \in \Ij \setminus \{1\}$, $s(\Etil_{\ul{i}} b) = s(b) + \alpha_{\ul{i}}$ if $\Etil_{\ul{i}}b \neq 0$.
\item For $i \in \Ij \setminus \{1\}$,
\begin{align}
\begin{split}
e_i \Gj(b) &= [\vphi_{\ul{i}}(b) + 1] \Gj(\Etil_{\ul{i}}b) + \sum_{\substack{b' \in \clB \setminus \{ \Etil_{\ul{i}}b \} \\ \wtj(b') = \wtj(b) + \gamma_i \AND s(b') \leq s(b) + \alpha_{\ul{i}}}} e_{b',b}^{(i)} \Gj(b'), \\
f_i \Gj(b) &= [\vphi_{-\ul{i}}(b) + 1] \Gj(\Etil_{-\ul{i}}b) + \sum_{\substack{b' \in \clB \setminus \{ \Etil_{-\ul{i}}b \} \\ \wtj(b') = \wtj(b) - \gamma_i \AND s(b') \leq s(b) + \alpha_{-\ul{i}}}} f_{b',b}^{(i)} \Gj(b')
\end{split} \nonumber
\end{align}
for some $e_{b',b}^{(i)}, f_{b',b}^{(i)} \in \A$.
\end{enumerate}
\end{defi}

\begin{lem}\label{condition ast}
Let $M \in \Oj$, and $\clL,\clB,\psij_M,M_\A$ as above.
\begin{enumerate}
\item If $r = 1$, then $M$ has the property $(\ast)$.
\item If $M \in \Oint$ and the global $\jmath$-crystal basis is the $\jmath$-canonical basis, then $M$ has the property $(\ast)$.
\end{enumerate}
\end{lem}

\begin{proof}
Setting $S$ and $s$ to be $\Lm$ and $\wt$, respectively, part $(1)$ is obvious, and part $(2)$ follows from Proposition \ref{str const for Gj-basis}.
\end{proof}

The main result in this paper is the following:
\begin{theo}\label{filtration}
Suppose that $M$ has the property $(\ast)$. Then, for each $\bflm \in \Pj(M)$, the following hold:
\begin{enumerate}
\item $W_{\succeq \bflm}(M)$ has a global $\jmath$-crystal basis $W_{\succeq \bflm}(\Gj(\clB)) := \{ \Gj(b) \mid I(b) \succeq \bflm \}$ with the associated balanced triple $(W_{\succeq \bflm}(\clL),W_{\succeq \bflm}(M_{\A}),W_{\succeq \bflm}(\psij_M(\clL)))$, where $W_{\succeq \bflm}(\clL) := W_{\succeq \bflm}(M) \cap \clL$, and so on.
\item $W_{\succ \bflm}(M)$ has a global $\jmath$-crystal basis $W_{\succ \bflm}(\Gj(\clB)) := \{ \Gj(b) \mid I(b) \succ \bflm \}$ with the associated balanced triple $(W_{\succ \bflm}(\clL),W_{\succ \bflm}(M_{\A}),W_{\succ \bflm}(\psij_M(\clL)))$, where $W_{\succ \bflm}(\clL) := W_{\succ \bflm}(M) \cap \clL$, and so on.
\item $W_{\bflm}(M)$ has a global $\jmath$-crystal basis $W_{\bflm}(\Gj(\clB)) := \{ \Gj(b) + W_{\succ \bflm}(M) \mid I(b) = \bflm \}$ with the associated balanced triple $(W_{\bflm}(\clL),W_{\bflm}(M_{\A}),W_{\bflm}(\psij_M(\clL)))$, where $W_{\bflm}(\clL) := W_{\succeq \bflm}(\clL)/W_{\succ \bflm}(\clL)$, and so on.
\item There exists a $\Uj$-module isomorphism $\xi: L(\bflm)^{\oplus m_{\bflm}} \rightarrow W_\bflm(M)$ which induces an isomorphism
$$
(\clL(\bflm)^{\oplus m_\bflm}, (L(\bflm)_{\A})^{\oplus m_\bflm}, \psij_{\bflm}(\clL(\bflm))^{\oplus m_\bflm}) \simeq (W_{\bflm}(\clL),W_{\bflm}(M_{\A}),W_{\bflm}(\psij_M(\clL))),
$$
where $m_\bflm := \dim \Hom_{\Uj}(L(\bflm),M)$ denotes the multiplicity of $L(\bflm)$ in $M$.
\end{enumerate}
\end{theo}

The proof will be given in Section \ref{sec:proof of main theorem}.

\begin{cor}
Let $\bflm \in \Pj$. Then, $\Gjlow(\bflm)$ is a unique global $\jmath$-crystal basis of $L(\bflm)$ satisfying the property $(*)$.
\end{cor}

\subsection{Operators $\tilde{e}_{i^+}$ and $\tilde{f}_{i^+}$}
The definitions of $\etil_{i'}$ and $\ftil_{i'}$ given in \cite{W17} are artificial, namely, they are defined by means of a distinguished basis $\Gjlow(\bflm)$, $\bflm \in \Pj$ (in \cite{W17}, it is denoted by $\{ b_T \mid T \in \clB(\bflm) \}$). Here, we define new operators $\etil_{i^+}$ and $\ftil_{i^+}$ for $i \in \Ij \setminus \{1\}$, and then, explain that the operators $\etil_{i'}$ and $\ftil_{i'}$ on $\jmath$-crystal bases are in fact intrinsic.

\begin{lem}
Let $r \geq 2$, $\bflm \in \Pj$, and consider the irreducible highest weight module $L(\bflm)$. As a $\Uj_{r-1}$-module, $L(\bflm)$ is multiplicity-free.
\end{lem}

\begin{proof}
Let $b \in \clB(\bflm)$ be a $\Uj_{r-1}$-highest weight vector with highest weight, say, $\bfmu \in \Pj_{r-1}$. If we identify $\clB(\bflm)$ with $\SST(\bflm)$, we have $T^\jmath_b \downarrow_{r-1} = T_{\bfmu}$. Since the entries of the boxes of $T^\jmath_{b'}$ corresponding to $\bflm/\bfmu$ are either $-r$ or $r$, it must hold that $\bflm/\bfmu$ is a horizontal strip. Conversely, given $\bfmu \in \Pj_{r-1}$ such that $\bflm/\bfmu$ is a horizontal strip, there exists a unique $b \in \clB(\bflm)$ which is a $\Uj_{r-1}$-highest weight vector with highest weight $\bfmu$. This proves the lemma.
\end{proof}

\begin{lem}\label{e_rprime}
Let $r \geq 2$, $\bflm \in \Pj$. Let $b \in \clB(\bflm)$ be such that $\etil_{r'} b \neq 0$. Then, there exist unique $b' \in \clB(\bflm)$ and $j \in \Ij \setminus \{1\}$ satisfying the following:
\begin{itemize}
\item $b'$ is a $\Uj_{r-1}$-highest weight vector.
\item There exist unique $\vep_i \in \{ \emptyset, \prime \}$ for each $j \leq i \leq r-1$ such that $b = \ftil_{r'} \ftil_{(r-1)^{\vep_{r-1}}} \cdots \ftil_{j^{\vep_j}} b'$.
\end{itemize}
\end{lem}

\begin{proof}
By the definition of $\etil_{r'}$, $b$ is a $\Uj_{r-1}$-highest weight vector with highest weight, say, $\bfmu \in \Pj_{r-1}$ such that $(T^\jmath_b)^- = T_\bflm^-$. Then, $T^\jmath_{\etil_{r'} b} \downarrow_{r-1}$ is obtained from $T_\bfmu$ by adding a box $\fbox{$r-1$}$ to the $(j-1)$-th row for some uniquely determined $j \in \Ij \setminus \{1\}$. Set $b_{r-1} := \etil_{r'} b$. Now, we have exactly one of the following; $\etil_{r-1} b_{r-1} \neq 0$ or $\etil_{(r-1)'} b_{r-1} \neq 0$. Choose a unique $\vep_{r-1} \in \{ \emptyset,\prime \}$ in a way such that $b_{r-2} := \etil_{(r-1)^{\vep_{r-1}}} b_{r-1} \neq 0$. Then, $T^\jmath_{b_{r-2}} \downarrow_{r-1}$ is obtained from $T_\bfmu$ by adding a box $\fbox{$r-2$}$ to the $(j-1)$-th row. Repeating this procedure, we obtain $\vep_{i} \in \{ \emptyset,\prime \}$ and $b_{i-1} \in \clB(\bflm)$ for $j \leq i \leq r-1$. By the construction, $T^\jmath_{b_{j-1}} \downarrow_{r-1}$ is obtained from $T_\bfmu$ by adding a box $\fbox{$j-1$}$ to the $(j-1)$-th row, which turned out to be $T_{\bfmu'}$, where $\bfmu' \in \Pj_{r-1}$ such that $\bfmu'_k = \bfmu_k + \delta_{k,j-1}$, $k \in \{ -(r-1),\ldots,r-1 \}$. Hence, $b_{j-1}$ is a $\Uj_{r-1}$-highest weight vector, and we have $b = \ftil_{r'} \ftil_{(r-1)^{\vep_{r-1}}} \cdots \ftil_{j^{\vep_j}} b_{j-1}$. This proves the assertion.
\end{proof}

Set $E_r(\bflm) := \{ \bfmu \in \Pj_{r-1} \mid \bfmu^- = \bflm^- \downarrow_{r-1} \AND \bflm^+/\bfmu^+ \text{ is a horizontal strip} \}$. Then, the assignment
$$
\{ b \in \clB(\bflm) \mid \etil_{r'}b \neq 0 \} \rightarrow E_r(\bflm);\ b \mapsto I^\jmath_{r-1}(b)
$$
is bijective. To each $\bfmu \in E_r(\bflm)$, we associate $b,b' \in \clB(\bflm)$, $j \in \Ij \setminus \{ 1 \}$, and $\vep_i \in \{ \emptyset,\prime \}$, $j \leq i \leq r-1$ as in Lemma \ref{e_rprime}.

Let $r \geq 2$. We define operators $\etil_{l^+}$ and $\ftil_{l^+}$ on every $\Uj$-modules in $\Oj$ inductively for all $2 \leq l < r$. Let $\bflm \in \Pj$. We define the linear operator $\etil_{r^+}$ on $L(\bflm)$ by
$$
\etil_{r^+} := \bigoplus_{\bfmu \in E_r(\bflm)}p_2(\bfmu) \circ \frac{1}{[\vphi_{\ul{r}}(b_\bfmu)+1]}e_r \circ p_1(\bfmu),
$$
where $b_\bfmu \in \clB(\bflm)$ is the corresponding element to $\bfmu \in E_r(\bflm)$, $p_1(\bfmu)$ is the projection from $L(\bflm)$ to the one-dimensional subspace $L(\bfmu)_{\wtj(\bfmu)}$;
$$
L(\bfmu)_{\wtj(\bfmu)} \subset L(\bfmu) \underset{\text{multiplicity free}}{\hookrightarrow} L(\bflm),
$$
and $p_2(\bfmu)$ is the projection from $L(\bflm)$ to the one-dimensional subspace $\ftil_{(r-1)^{\delta_{r-1}}} \cdots \ftil_{j^{\delta_j}} L(\bfmu')_{\wtj(\bfmu')}$;
$$
\ftil_{(r-1)^{\delta_{r-1}}} \cdots \ftil_{j^{\delta_j}} L(\bfmu')_{\wtj(\bfmu')} \subset L(\bfmu') \underset{\text{multiplicity free}}{\hookrightarrow} L(\bflm),
$$
where $\delta_l = \emptyset$ if $\vep_l = \emptyset$, and $\delta_l = +$ if $\vep_l = \prime$ for $l = j,\ldots,r-1$. Also, we define $\ftil_{r^+}$ by
$$
\ftil_{r^+} = \bigoplus_{\bfmu \in E_r(\bflm)} \etil_{r^+}\inv \circ p_2(\bfmu),
$$
where $\etil_{r^+}\inv$ is the inverse of the linear isomorphism $\etil_{r^+} : L(\bfmu)_{\wtj(\bfmu)} \rightarrow \ftil_{(r-1)^{\delta_{r-1}}} \cdots \ftil_{j^{\delta_j}} L(\bfmu')_{\wtj(\bfmu')}$.
Finally, we extend the definitions of $\etil_{r^+}$ and $\ftil_{r^+}$ to a general $\Uj$-module $M \in \Oj$ by the complete reducibility of $M$.

\begin{prop}
Let $\bflm \in \Pj$ and $v \in L(\bflm)$ a highest weight vector. Then, we have
\begin{align}
\begin{split}
\clL(\bflm) &= \Span_{\Ao} \{ \ftil_{i_1} \cdots \ftil_{i,l}v \mid l \in \Z_{\geq 0},\ i_1,\ldots,i_l \in \Ij \sqcup \{ 2^+,\ldots,r^+ \} \}, \\
\clB(\bflm) &= \{ \ftil_{i_1} \cdots \ftil_{i,l}v + q\clL(\bflm) \mid l \in \Z_{\geq 0},\ i_1,\ldots,i_l \in \Ij \sqcup \{ 2^+,\ldots,r^+ \} \} \setminus \{0\}.
\end{split} \nonumber
\end{align}
Moreover, on $\clB(\bflm)$, we have $\etil_{i'} = \etil_{i^+}$ and $\ftil_{i'} = \ftil_{i^+}$ for all $i \in \Ij \setminus \{1\}$.
\end{prop}

\begin{proof}
We proceed by induction on $r$. Assume that the assertion holds for all $2 \leq l < r$ (we assume nothing when $r=2$). Let $\bfmu \in E_r(\bflm)$ and $b_\bfmu$, $b'$, $\bfmu'$ be as above. By the uniqueness of the $\jmath$-crystal bases for $\Uj_{r-1}$-modules, there exists a unique $v_\bfmu \in \clL(\bflm)$ such that $\Uj_{r-1}v_\bfmu = L(\bfmu)$, $v_\bfmu + q\clL(\bflm) = b_\bfmu$. Then, we can write
$$
v_\bfmu = \Gjlow(b_\bfmu) + \sum_{b' \in \clB(\bflm) \setminus \{ b_\bfmu \}} a_{b'} \Gjlow(b')
$$
for some $a_{b'} \in q\Ao$. Note that this equation implies that $\etil_{r'} (v_\bfmu) \in \Gjlow(\etil_{r'} b_\bfmu) + q\clL(\bflm)$. Also, we have
$$
\frac{1}{[\vphi_{\ul{r}}(b_\bfmu)+1]} e_r v_\bfmu = \Gjlow(\etil_{r'}b_\bfmu) + \sum_{b' \in \clB(\bflm)} c_{b'} \Gjlow(b') \qu (\text{since $\etil_{r'} b_\bfmu = \Etil_{\ul{r}} b_\bfmu$.})
$$
for some $c_{b'} \in \A$. Again, by the complete reducibility of the $\Uj_{r-1}$-crystal bases, there exists a unique $v_{\bfmu'} \in \clL(\bflm)$ such that $\Uj_{r-1} v_{\bfmu'} = L(\bfmu')$, $v_{\bfmu'} + q\clL(\bflm) = b'$. By our induction hypothesis, we have $u := \ftil_{(r-1)^{\delta_{r-1}}} \cdots \ftil_{j^{\delta_j}}(v_{\bfmu'}) \in \clL(\bflm) \cap \Uj_{r-1} v_{\bfmu'}$ and $u + q\clL(\bflm) = \etil_{r'}b_\bfmu$. Then, we can write
$$
u = \Gjlow(\etil_{r'}b_\bfmu) + \sum_{b' \in \clB(\bflm)} d_{b'} \Gjlow(b')
$$
for some $d_{b'} \in q\Ao$. Hence, we have
$$
\etil_{r^+}(v_\bfmu) \in \Gjlow(\etil_{r'}b_\bfmu) + q\clL(\bflm).
$$
Since we took $\bfmu \in E_r(\bflm)$ arbitrarily, this equation ensures that $\etil_{r^+}$ preserves $\clL(\bflm)$ and $\clB(\bflm) \sqcup \{0\}$, and that $\etil_{r^+} = \etil_{r'}$ on $\clB(\bflm)$. By the definition of $\ftil_{r^+}$, it also preserves $\clL(\bflm)$ and $\clB(\bflm) \sqcup \{0\}$, and coincides with $\ftil_{r'}$ on $\clB(\bflm)$. Now, the assertions are clear by the definition of $(\clL(\bflm),\clB(\bflm))$.
\end{proof}

\begin{cor}
Let $M \in \Oj$ be a $\Uj$-module with a $\jmath$-crystal basis $(\clL,\clB)$. Then $\etil_{i'} = \etil_{i^+}$ and $\ftil_{i'} = \ftil_{i^+}$ on $\clB$ for all $i \in \Ij \setminus \{1\}$.
\end{cor}

\section{Proof of Theorem \ref{filtration}}\label{sec:proof of main theorem}
For a $\Uj$-module $M$ with a global $\jmath$-crystal basis $\Gj(\clB)$, and for $m \in M$, $b \in \clB$, let $[m:\Gj(b)]$ denote the coefficient of $\Gj(b)$ in $m$.

\subsection{The $r=1$ case}
In this subsection, we prove Theorem \ref{filtration} for $r=1$.

\begin{proof of filtration}
We proceed by descending induction on $\bflm$ with respect to $\preceq$. Assume that the statement holds for all $\bflm' \succ \bflm$. Replacing $M$ with $M/W_{\succ \bflm}(M)$, we may assume that $\bflm$ is maximal among $\Pj(M)$. Let $b_1,\ldots,b_{m_\bflm} \in \clB$ and $u_1,\ldots,u_{m_\bflm} \in \clL$ be distinct highest weight vectors of type $\bflm$ with $u_i + q\clL = b_i$, $i = 1,\ldots,m_\bflm$. By retaking the $u_i$'s if necessary, we may assume that $[u_i:\Gj(b_j)] = \delta_{i,j}$ for all $i,j$. Fix $i$ arbitrarily, and set $b := b_i$, $u := u_i$. Then, we can write
$$
u = \Gj(b) + \sum_{\substack{b' \\ I^\jmath(b') \not\succeq \bflm}} c_{b'}\Gj(b'), \qu c_{b'} \in q\Ao.
$$
We first prove that $c_{b'} = 0$ for all $b'$ with $\vep_1(b') = 0$. Assume contrary, and take $b' \in \clB \setminus \{ b \}$ such that $c_{b'} \neq 0$, $\vep_1(b') = 0$, and $\vphi_1(b')$ is minimal among $\{ \vphi_1(b'') \mid c_{b''} \neq 0,\ \vep_1(b'') = 0 \}$. Set $\bfmu := I^\jmath(b')$. Then, we have $\wtj(\bfmu) = \wtj(\bflm)$, in particular, $\bfmu_0 = \bflm_0$. Since $\bfmu \not\succeq \bflm$, we have $\vphi_1(b') = \bfmu_0-\bfmu_{-1} > \bflm_0-\bflm_{-1} = \vphi_1(b)$. Hence, we have
\begin{align}
\begin{split}
-f_1^{(\vphi_1(b)+1)}\Gj(b) &= c_{b'} \bigl( \Gj(\ftil_1^{\vphi_1(b)+1}b') + \sum_{\substack{b'' \\ \vep_1(b'') > \vphi_1(b)+1}} d_{b'',b'} \Gj(b'') \bigr) \\
&\qu + \sum_{b''' \neq b'} \sum_{\substack{b'''' \\ \vep_1(b'''') \geq \vep_1(b''') + \vphi_1(b)+1}} d_{b'''',b'''} \Gj(b''''),
\end{split} \nonumber
\end{align}
for some $d_{b_1,b_2} \in \A$. By our assumption, the coefficient of $\Gj(\ftil_1^{\vphi_1(b)+1}b')$ in the right-hand side is equal to $c_{b'}$. On the other hand, the left-hand side is fixed by $\psij_M$, and it belongs to $M_{\A}$. Therefore, we have $c_{b'} \in q\A_0 \cap \A$ and $\ol{c_{b'}} = c_{b'}$, which implies $c_{b'} = 0$.

Next, we prove that $c_{b'} = 0$ for all $b'$ with $\vep_1(b') > 0$. Assume contrary that $c_{b'} \neq 0$ for some such $b'$. Set $\bfmu := I^{\jmath}(b')$. Since $\bflm$ is maximal, we have $\bfmu_0 + \bfmu_{-1} < \bflm_0 + \bflm_{-1}$. Substituting $\bfmu_0 = \bflm_0 + \vep_1(b')$, $\bfmu_{-1} = \bflm_0 - \vphi_1(b')$, and $\bflm_0 - \bflm_{-1} = \vphi_1(b)$, we obtain $\vphi_1(b') > \vphi_1(b) + \vep_1(b')$. We may assume that $(\vep_1(b'),\vphi_1(b'))$ is minimal (with respect to the lexigographical order) among such $b'$'s. Then, for all $t = 1,\ldots, \vep_1(b')+1$, we have
\begin{align}
-f_1^{(\vphi_1(b)+t)}\Gj(b) = c_{b'}{\vep_1(b')+\vphi_1(b)+t \brack \vep_1(b')} \Gj(\ftil_1^{\vphi_1(b)+t} b') + \otherterms. \nonumber
\end{align}
This implies that $c_{b'} \in q\A_0$, $\ol{c_{b'}} = c_{b'}$, and $c_{b'}{\vep_1(b')+\vphi_1(b)+t \brack \vep_1(b')} \in \A$ for all $t = 1,\ldots, \vep_1(b')+1$. Now, it suffices to show that $c_{b'} \in \A$, which follows from next lemma.

This far, we have proved that $\Gj(b) = u$, and hence, we have $e_1 \Gj(b) = 0$ and $\Uj_1 \Gj(b) \simeq L(\bflm)$. Then, for all $n = 1,\ldots, \bflm_0-\bflm_{-1}$, we have
\begin{align}
f_1^{(n)} \Gj(b) = f_1^{(n)} u = \ftil_1^n u. \nonumber
\end{align}
The left-hand side belongs to $M_{\A}$, while the right-hand side belongs to $\clL$. Moreover, we have $\psij_M(f_1^{(n)} \Gj(b)) = f_1^{(n)} \Gj(b)$, and $\ftil_1^n u + q\clL = \ftil_1^n b$. This implies that $f_1^{(n)} \Gj(b) = \Gj(\ftil_1^n b)$. Thus, the proof completes.
\end{proof of filtration}

\begin{lem}\label{lemma}
Let $A \in \Q(p,q)$, $m \geq n \in \Z_{\geq 0}$. Suppose that $A {m+t \brack n} \in \A$ for all $t = 1,\ldots, n+1$. Then, we have $A \in \A$.
\end{lem}

\begin{proof}
Let us write $A = B/C$ for some $B,C \in \A_0 \cap \A$ that are coprime. By the hypothesis, $C$ is a common devisor of ${m+t \brack n}$, $t = 1,\ldots,n+1$. Hence, it suffices to show that the greatest common divisor of them in $\Z[q]$ is equal to $1$. This is equivalent to say that the greatest common divisor of $a_t := [m+t][m+t-1] \cdots [m+t-n+1]$, $t = 1,\ldots,n+1$ is equal to $[n]!$. Since $[l] = q^{-l} \prod_{1 \neq d | l} \Phi_d$, where $\Phi_d = \Phi_d(q^2)$ denotes the $d$-th cyclotomic polynomial in variable $q^2$, we have
$$
b_t := q^{n(m+t)-\frac{n(n-1)}{2}}a_t = \prod_{l=0}^{n-1} \prod_{1 \neq d | (m+t-l)} \Phi_d,
$$
which is the irreducible decomposition of $b_t$ in $\Z[q^2]$. Then, we have
$$
b_t = \prod_{d \geq 2} \Phi_d^{m_{d,t}}, \text{ where } m_{d,t} := |\{ 0 \leq l \leq n-1 \mid d | (m+t-l) \}|,
$$
and hence,
$$
\gcd_{1 \leq t \leq n+1}(b_t) = \prod_{d \geq 2} \Phi_d^{\min_{1 \leq t \leq n+1}(m_{d,t})}.
$$
We prove that $\min_{1 \leq t \leq n+1}(m_{d,t}) = \lfloor \frac{n}{d} \rfloor$ for all $d$. It is clear that $m_{d,t} \geq \lfloor \frac{n}{d} \rfloor$ for all $t$ since $\{ m+t,m+t-1,\ldots,m+t-(\lfloor \frac{n}{d} \rfloor d-1) \}$ contains exactly $\lfloor \frac{n}{d} \rfloor$ integers divisible by $d$. If $\min_{1 \leq t \leq n+1}(m_{d,t}) > \lfloor \frac{n}{d} \rfloor$, then $\{ m+t-\lfloor \frac{n}{d} \rfloor d, m+t-(\lfloor \frac{n}{d} \rfloor d+1), \ldots, m+t-(n-1) \}$ contains at least one multiple of $d$ for all $t$. Then, for $t=1$, there exists $l_1 \in \{ \lfloor \frac{n}{d} \rfloor d, \lfloor \frac{n}{d} \rfloor d+1,\ldots,n-1 \}$ such that $m+1-(\lfloor \frac{n}{d} \rfloor d + l_1) \in d\Z$. Set $t' := n-l_1 + 1$, and consider the integers
$$
m+t'-\lfloor \frac{n}{d} \rfloor d, m+t'-(\lfloor \frac{n}{d} \rfloor d+1), \ldots, m+t'-(n-1) = (m+1-l_1) + 1.
$$
These are $(n-\lfloor \frac{n}{d} \rfloor d)$ consecutive integers with $(m+1-l_1) + 1 = 1$ modulo $d$. Since $n-\lfloor \frac{n}{d} \rfloor d < d$, they have no multiples of $d$. Hence, we have $\min_{1 \leq t \leq n+1}(m_{d,t}) = \lfloor \frac{n}{d} \rfloor$ for all $d \geq 2$. Thus, we obtain
$$
\gcd_{1 \leq t \leq n+1}(b_t) = \prod_{d \geq 2} \Phi_d^{\lfloor \frac{n}{d} \rfloor} = \prod_{d=2}^n \Phi_d^{\lfloor \frac{n}{d} \rfloor} = \prod_{l=2}^n \left(\prod_{1 \neq d' | l} \Phi_{d'}\right) = \prod_{l=2}^n [l] = [n]!.
$$
This proves the lemma.
\end{proof}

\subsection{The $r \geq 2$ case}
Now, we are ready to prove Theorem \ref{filtration} by induction on $r$.

When $r=1$, we have already completed the proof. Let $r \geq 2$ and assume that the assertions hold for all $r' < r$.

\begin{lem}\label{weak lemma}
Let $\bflm \in \Pj(M)$ be a maximal element, $b \in \clB$ such that $I^\jmath(b) = \bflm$ and $\etil_i b = 0$ for all $i \in \Ij$. Suppose the following:
\begin{enumerate}
\item There exists a homomorphism $\xi : L(\bflm) \rightarrow M$ of $\Uj$-modules such that $\xi(\Gjlow(\Tj_{b'})) = \Gj(b')$ for all $b' \in \Cj(b)$ which is strongly connected to some $b'' \in \Cj(b)$ with $\wtj(b) <^\jmath \wtj(b'')$.
\item $\xi$ commutes with the $\psij$-involutions on $L(\bflm)$ and $M$.
\item $[\xi(\Gjlow(\Tj_b)) : \Gj(b)] = 1$.
\end{enumerate}
Then, we have
$$
\xi(\Gjlow(\Tj_b)) = \Gj(b) + \sum_{\substack{b' \in \clB \setminus \{ b \} \\ \Tj_{b'} = \Tj_b}} c_{b'} \Gj(b') + \sum_{\substack{b'' \in \Cj(b'),\ c_{b' \neq 0} \\ s(b'') < s(b')}} c_{b''} \Gj(b'').
$$
for some $c_{b'},c_{b''} \in \Ao$.
\end{lem}

\begin{proof}
Since $\Uj$-module homomorphisms preserve $\jmath$-crystal lattices, we have $\xi(\Gjlow(\Tj_b)) \in \clL$, and $\xi(\Gjlow(\Tj_b)) + q\clL = b$. Let us write
$$
\xi(\Gjlow(\Tj_b)) = \Gj(b) + \sum_{b' \in \clB \setminus \{b\}} c_{b'} \Gj(b')
$$
for some $c_{b'} \in q\Ao$. Also, since $\xi$ commutes with $\psij$-involutions, we have $\ol{c_b} = c_b$, $\ol{c_{b'}} = c_{b'}$. We claim the following: if $b' \in \clB \setminus \{b\}$ satisfies
$$
(\dagger) \qu \text{$c_{b'} \neq 0$ and $s(b')$ is maximal among $\{ s(b'') \mid b'' \in \clB \setminus \{b\} \AND c_{b''} \neq 0 \}$},
$$
then $\Tj_{b'}(-i) \geq \bflm_{-i}$ for all $i = 0,1,\ldots,r$. By the case $r = 1$, we have $I^\jmath_1(b') \succeq I^\jmath_1(b)$, which implies $\Tj_{b'}(0) = \Tj_{b}(0) = \bflm_0$, and $\Tj_{b'}(-1) \geq \Tj_b(-1) = \bflm_{-1}$. We proceed by induction on $i$. Assume that $i \geq 2$, and that $\Tj_{b'}(-(i-1)) \geq \bflm_{-(i-1)}$ for all $b'$ satisfying $(\dagger)$. Suppose that there exists $b'$ satisfying $(\dagger)$ such that $\Tj_{b'}(-i) < \bflm_{-i}$. Let $b'' \in \clB \setminus \{b\}$ be such that $s(b'') = s(b')$ and $\vphi_{-\ul{b''}}$ is minimal among such elements. Recall that $s(b'') = s(b')$ implies $\wt(b'') = \wt(b')$, and hence, $\Tj_{b''}(-i) = \Tj_{b'}(-i) < \bflm_{-i}$. Then, we have
$$
\vep_{-\ul{i}}(b'') = \vphi_{-\ul{i}}(b'') + \Tj_{b''}(-(i-1)) - \Tj_{b''}(-i) > \Tj_{b''}(-(i-1)) - \bflm_{-i} + \vphi_{-\ul{i}}(b'').
$$
By the minimality of $\vphi_{-\ul{i}}(b'')$, it holds that
$$
[f_i^{(t)} \sum_{b' \in \clB \setminus \{b\}} c_{b''} \Gj(b') : \Gj(\ftil_i^t b'')] = c_{b'} {t \brack \vphi_{-\ul{i}}(b'')} \neq 0
$$
for all $\Tj_{b''}(-(i-1)) - \bflm_{-i} +1 \leq t \leq \Tj_{b''}(-(i-1)) - \bflm_{-i} + \vphi_{-\ul{i}}(b'') + 1$. On the other hand, $f_i^{(t)} \Gjlow(\Tj_b)$ is the sum of $\Gjlow(\Tj_{\ftil_i^t b})$ and an $\A$-linear combination of $\Gjlow(\Tj_{\bhat})$ such that $\bhat \in \Cj(b_t)$ is strongly connected to $b''' \in \Cj(b_t)$ with $\wtj(b) <^\jmath \wtj(b''')$. Hence, we have
\begin{align}
\begin{split}
\xi(\Gjlow(\Tj_{\ftil_i^t b})) &= f_i^{(t)} \xi(\Gjlow(\Tj_b)) + \sum_{\bhat} a_{\bhat} \Gj(\bhat) \\
&= f_i^{(t)} \Gj(b) + f_i^{(t)} \sum_{b' \in \clB \setminus \{b\}} c_{b'} \Gj(b') + \sum_{\bhat} a_{\bhat} \Gj(\bhat)
\end{split} \nonumber
\end{align}
for some $a_{\bhat} \in \A$. Here, note that we have $\etil_j \ftil_i^tb = 0$ for all $j = 1,\ldots,i-1$, $[\xi(\Gjlow(\Tj_{\ftil_i^t b})) : \Gj(\ftil_i^t b)] = 1$, and $s(\ftil_i^t b'')$ is maximal among $\{ s(b''') \mid b''' \neq \ftil_i^tb \AND [\xi(\Gjlow(\Tj_{\ftil_i^t b})) : \Gj(b''')] \neq 0 \}$. Then, by our induction hypothesis on $i$, we obtain that $\Tj_{\ftil_i^t b''}(-(i-1)) \geq \Tj_{\ftil_i^t b}(-(i-1)) = \bflm_{-i}$, which is a contradiction since $t \geq \Tj_{b''}(-(i-1)) - \bflm_{-i} + 1$. Hence we must have $[\xi(\Gjlow(\Tj_{\ftil_i^t b})) : \Gj(\ftil_i^t b'')] = 0$. Since
$$
[\xi(\Gjlow(\Tj_{\ftil_i^t b})) : \Gj(\ftil_i^t b'')] = c_{b''} {t \brack \vphi_{-\ul{i}}(b'')} + [f_i^{(t)} \Gj(b) : \Gj(\ftil_i^t b'')] + a_{\ftil_i^t b''},
$$
and the second and the third term of the right-hand side lies in $\A$, we obtain
$$
c_{b''} {t \brack \vphi_{-\ul{i}}(b'')} \in \A
$$
for all $\Tj_{b''}(-(i-1)) - \bflm_{-i} + 1 \leq t \leq \Tj_{b''}(-(i-1)) - \bflm_{-i} + \vphi_{-\ul{i}}(b'') + 1$. By Lemma \ref{lemma}, this implies $c_{b''} = 0$.

This far, we have proved that if $b' \in \clB \setminus \{b\}$ satisfies $(\dagger)$, then we have $\Tj_{b'}(-i) \geq \bflm_{-i}$ for all $i \in \{ 0,1,\ldots,r \}$. In particular, we have $\Ij(b') = \bflm$ for such $b'$ (since $\bflm$ is maximal in $\Pj(M)$). In this case, the condition $\Tj_{b'}(-i) \geq \bflm_{-i}$ for all $i$ forces $b'$ to satisfy that $\Tj_{b'} = \Tj_b$. Hence, we have
$$
\xi(\Gjlow(\Tj_b)) = \Gj(b) + \sum_{\substack{b' \in \clB \setminus \{ b \} \\ \Tj_{b'} = \Tj_b}} c_{b'} \Gj(b') + \sum_{\substack{b'' \in \Cj(b'),\ c_{b' \neq 0} \\ s(b'') < s(b')}} c_{b''} \Gj(b''),
$$
as desired.
\end{proof}

\begin{lem}\label{strong lemma}
Let $\bflm \in \Pj(M)$ be a maximal element, $j \in \Ij \setminus \{ 1 \}$, $b \in \clB$ such that $I^\jmath(b) = \bflm$, $\etil_i b = 0$ for all $i \in \Ij$, and $\etil_{j'}(b) \neq 0$. Suppose the following:
\begin{enumerate}
\item There exists a homomorphism $\xi : L(\bflm) \rightarrow M$ of $\Uj$-modules such that $\xi(\Gjlow(\Tj_{b'})) = \Gj(b')$ for all $b' \in \Cj(b)$ which is strongly connected to some $b'' \in \Cj(b)$ with $\wtj(b) <^\jmath \wtj(b'')$.
\item $\xi$ commutes with the $\psij$-involutions on $L(\bflm)$ and $M$.
\end{enumerate}
Then, we have
$$
\xi(\Gjlow(\Tj_b)) = \Gj(b) + \sum_{\substack{b' \in \clB \setminus \{ b \} \\ \Tj_{b'} = \Tj_b}} c_{b'} \Gj(b') + \sum_{\substack{b'' \in \Cj(b'),\ c_{b' \neq 0} \\ s(b'') < s(b')}} c_{b''} \Gj(b'').
$$
for some $c_{b'},c_{b''} \in \Ao$.
\end{lem}

\begin{proof}
If we can prove that $c_b := [\xi(\Gjlow(\Tj_b)) : \Gj(b)] = 1$, then the assertion follows from the previous lemma.  Hence, we aim to show $c_b = 1$.

By the same argument as before, we have
$$
[\xi(\Gjlow(\Tj_{\ftil_i^t b})) : \Gj(\ftil_i^t b'')] = c_{b''} {t \brack \vphi_{-\ul{i}}(b'')} + c_b[f_i^{(t)} \Gj(b) : \Gj(\ftil_i^t b'')] + a_{\ftil_i^t b''},
$$
for all $b'' \in \clB \setminus \{b\}$ satisfying $(\dagger)$. Here, let us assume further that $s(b') > s(b)$. Then, we have $[f_i^{(t)} \Gj(b) : \Gj(\ftil_i^t b'')] = 0$ since $f_i^{(t)} \Gj(b)$ is a linear combination of $\Gj(\btil)$ with $s(\btil) \leq s(b) + t\alpha_{-\ul{i}} < s(b'') + t \alpha_{-\ul{i}} = s(\ftil_i^t b'')$. Hence, we have $c_{b''}{t \brack \vphi_{-\ul{i}}(b'')} \in \A$, and therefore, $c_{b''} = 0$ by Lemma \ref{lemma}. In particular, we obtain that $s(b)$ is maximal. Then, we have
$$
[e_j(c_b \Gj(b) + \sum c_{b'} \Gj(b')) : \Gj(\etil_{j'}b)] = c_b[\vphi_{\ul{j}}(b) + 1].
$$
On the other hand, since $[e_j \Gjlow(\Tj_b) : \Gjlow(\etil_{j'}b)] = [\vphi_{\ul{j}}(b) + 1]$, we have
$$
[e_j(c_b \Gj(b) + \sum c_{b'} \Gj(b')) : \Gj(\etil_{j'}b)] = [\vphi_{\ul{j}}(b) + 1],
$$
and hence, $c_{b} = 1$, as desired.
\end{proof}

We prove Theorem \ref{filtration} by descending induction (with respect to $\preceq$) on $\bflm$. As in the $r=1$ case, we may assume that $\bflm$ is maximal among $P^\jmath(M)$. Then, in order to complete the proof, we have to show the following:
\begin{enumerate}
\item $I_\bflm(M)$ has a basis $\{ \Gj(b) \mid I^\jmath(b) = \bflm \}$.
\item There exists an isomorphism $\xi:L(\bflm)^{\oplus m_\bflm} \rightarrow I_\bflm(M)$ of $\Uj$-modules which sends the $\jmath$-global basis elements of $L(\bflm)^{\oplus m_\bflm}$ to those of $I_\bflm(M)$, where $m_\bflm$ denotes the multiplicity of $L(\bflm)$ in $M$.
\end{enumerate}

Let $b_1,\ldots,b_{m_\bflm} \in \clB$ and $u_1,\ldots,u_{m_\bflm} \in \clL$ be distinct highest weight vectors of type $\bflm$ with $u_t + q\clL = b_t$, $t = 1,\ldots,m_\bflm$. By retaking the $u_t$'s if necessary, we may assume that $[u_t:\Gj(b_u)] = \delta_{t,u}$ for all $t,u$. Let $\xi_t : L(\bflm) \rightarrow M$ be the $\Uj$-homomorphism which sends $v_\bflm$ to $u_t$.

\begin{lem}\label{lemma 1}
We have $\xi_t(\Gjlow(\Tj_{b_t})) = \Gj(b_t)$ for all $t = 1,\ldots,m_\bflm$.
\end{lem}

\begin{proof}
By the setting above, we can write
$$
\xi_t(\Gjlow(\Tj_{b_t})) = u_t = \Gj(b_t) + \sum_{\substack{b' \\ I^\jmath(b') \not\succeq \bflm}} c_{b'}\Gj(b'), \qu c_{b'} \in q\Ao.
$$
Then, we can apply Lemma \ref{weak lemma} to obtain $\xi_t(\Gjlow(\Tj_{b_t})) = \Gj(b_t)$ as desired.
\end{proof}

In order to complete the proof, it suffices to prove the following: For each $t = 1,\ldots,m_\bflm$ and $b \in \Cj(b_t)$, we have $\xi_t(\Gjlow(\Tj_b)) = \Gj(b)$. We prove this statement by descending induction on $\wtj(b)$ and $I^{\jmath}_{r-1}(b)$. When $\wtj(b)$ is maximal, it must hold that $b = b_t$, and in this case, we have already shown that $\xi_t(\Gjlow(\Tj_{b_t})) = \Gj(b_t)$. Suppose that $\wtj(b) <^\jmath \wtj(b_t)$, and the statement holds for all $b' \in \bigsqcup_{t=1}^{m_\bflm} \Cj(b_t)$ such that $\wtj(b') {}^\jmath > \wtj(b)$ or $\wtj(b') = \wtj(b)$ and $I^\jmath_{r-1}(b') \succ I^\jmath_{r-1}(b)$. In this case, since $b$ is not a $\Uj$-highest weight vector, the exists $i \in \IIj$ such that $\etil_i b \neq 0$.

\begin{lem}
Suppose there exists $i \in \Ij$ such that $\etil_i b \neq 0$. Then, the statement holds.
\end{lem}

\begin{proof}
Set $b' := \etil_i^{\vep_i(b)} b$. We prove the lemma by descending induction on $\vep_i(b')$. Since $\wtj(b') > \wtj(b)$, we have $\Gj(b') = \xi_t(\Gjlow(\Tj_{b'})) \in \Uj \Gj(b_i)$. We know that $\Gj(b)$ (resp., $\Gjlow(\Tj_b)$) is the sum of $\ftil_i^{(\vep_i(b))} \Gj(b')$ (resp., $\ftil_i^{(\vep_i(b))} \Gjlow(\Tj_{b'})$) and a $q\Q[q]$-linear combination of $\Gj(b'')$ (resp., $\Gjlow(\Tj_{b''})$) with $\wtj(b'') = \wtj(b)$ and $\vep_i(b'') > \vep_i(b)$. By our induction hypothesis, $\Gj(b) - \xi_t(\Gjlow(b))$ is a $q\Q[q]$-linear combination of $\Gj(b'')$'s, and is $\psij_M$-invariant. Such a vector must be zero, and hence, we obtain $\Gj(b) = \xi_t(\Gjlow(b))$.
\end{proof}

\begin{lem}\label{lemma 3}
Suppose there exists $j \in \Ij \setminus \{ 1 \}$ such that $\etil_{j'} b \neq 0$ and $\etil_i b = 0$ for all $i \in \Ij$. Then, the statement holds.
\end{lem}

\begin{proof}
Apply Lemma \ref{strong lemma}.
\end{proof}

Now, one can complete the proof by combining Lemma \ref{lemma 1}-\ref{lemma 3} since each $b \in \clB$ with $I^\jmath(b) = \bflm$ is connected to $b_t$ for some $t = 1,\ldots,m_\bflm$.


\begin{thebibliography}{99}
%\bibitem[BB05]{BB05} A. Bj\"{o}rner and F. Brenti, Combinatorics of Coxeter Groups, Graduate Texts in Mathematics, 231. Springer, New York, 2005. xiv+363 pp.

\bibitem[BW13]{BW13} H. Bao and W. Wang, A new approach to Kazhdan-Lusztig theory of type $B$ via quantum symmetric pairs, Ast\'{e}risque 402 (2018), 134pp., arXiv:1310.0103v2.

\bibitem[BW18]{BW18} H. Bao and W. Wang, Canonical bases arising from quantum symmetric pairs, Invent. Math. 213 (2018), no. 3, 1099--1177. 

\bibitem[BWW16]{BWW16} H. Bao, W. Wang, and H. Watanabe, Multiparameter quantum Schur duality of type $B$, Proc. Amer. Math. Soc. 146 (2018), 3203--3216.

\bibitem[BWW18]{BWW18} H. Bao, W. Wang, and H. Watanabe, Addendum to ``Canonical bases arising from quantum symmetric pairs'', arXiv:1808.09388v2.

\bibitem[BI03]{BI03} C. Bonnaf\'{e} and L. Iancu, Left cells in type $B_n$ with unequal parameters, Represent. Theory 7 (2003), 587--609.

\bibitem[DDPW08]{DDPW08} B. Deng, J. Du, B. Parshall, and J. Wang, Finite Dimensional Algebras and Quantum Groups, Mathematical Surveys and Monographs, 150, American Mathematical Society, Providence, RI, 2008, xxvi+759 pp.

\bibitem[Deo87]{Deo87} V. V. Deodhar, On some geometric aspects of Bruhat orderings. II. The parabolic analogue of Kazhdan-Lusztig polynomials, J. Algebra 111 (1987), no. 2, 483--506.

%\bibitem[D93]{D93} M. J. Dyer, Hecke algebras and shellings of Bruhat intervals, Compositio Math. 89 (1993), no. 1, 91-115. 

%\bibitem[GAP16]{GAP16} The GAP group, GAP - Groups, algorithms, and programming, Version 4.8.3, 2016.

%\bibitem[G03]{G03} M. Geck, On the induction of Kazhdan-Lusztig cells, Bull. London Math. Soc. 35 (2003), no. 5, 608-614.

%\bibitem[GL92]{GL92} I. Grojnowski and G. Lusztig, On bases of irreducible representations of quantum $GL_n$, Kazhdan-Lusztig theory and related topics (Chicago, IL, 1989), 167--174, Contemp. Math., 139, Amer. Math. Soc., Providence, RI, 1992. 

\bibitem[HK02]{HK02} J. Hong and S.-J. Kang, Introduction to Quantum Groups and Crystal Bases, Graduate Studies in Mathematics, 42. American Mathematical Society, Providence, RI, 2002. xviii+307 pp.

%\bibitem[HY03]{HY03} R. Howlett and Y. Yin, Inducing $W$-graphs, Math. Z. 244 (2003), no. 2, 415-431.

%\bibitem[J86]{J86} M. Jimbo, A q-analogue of $U(\mathfrak{gl}(N + 1))$, Hecke algebras, and Yang-Baxter equation, Lett. Math. Phys., 11 (1986), no. 3, pp. 247-252. 

\bibitem[K90]{K90} M. Kashiwara, Crystalizing the 
$q$-analogue of universal enveloping algebras, Comm. Math. Phys. 133 (1990), no. 2, 249--260. 

\bibitem[K93]{K93} M. Kashiwara, Global crystal bases of quantum groups, Duke Math. J. 69 (1993), no. 2, 455--485. 

\bibitem[K02]{K02} M. Kashiwara, On level-zero representations of quantized affine algebras, Duke Math. J. 112 (2002), no. 1, 117--175. 

\bibitem[KL79]{KL79} D. Kazhdan and G. Lusztig, Representations of Coxeter groups and Hecke algebras, Invent. Math. 53 (1979), no. 2, 165--184. 

%\bibitem[KP11]{KP11} S. Kolb and J. Pellegrini, Braid group actions on coideal subalgebras of quantized enveloping algebras, J. Algebra 336 (2011), 395-416. 

\bibitem[Le99]{Le99} G. Letzter, Symmetric pairs for quantized enveloping algebras, J. Algebra 220 (1999), no. 2, 729--767. 

%\bibitem[LS91]{LS91} S. Z. Levendorskii and Y. S. Soibelman, The quantum Weyl group and a multiplicative formula for the $R$-matrix of a simple Lie algebra, Funct. Anal. Appl. 25 (1991), no. 2, 143-145 .

\bibitem[L90a]{L90a} G. Lusztig, Canonical bases arising from quantized enveloping algebras, J. Amer. Math. Soc. 3 (1990), no. 2, 447--498. 

\bibitem[L90b]{L90} G. Lusztig, Quantum groups at roots of $1$, Geom. Dedicata, 33 (1990), 89--113.

\bibitem[L93]{L94} G. Lusztig, Introduction to Quantum Groups, Reprint of the 1994 edition, Modern Birkh\"{a}user Classics, Birkh\"{a}user/Springer, New York, 2010.

\bibitem[L03]{L} G. Lusztig, Hecke algebras with unequal parameters, CRM Monograph Series, 18. American Mathematical Society, Providence, RI, 2003. vi+136 pp.

%\bibitem[S97]{S97} W. Soergel, Kazhdan-Lusztig polynomials and a combinatoric for tilting modules, Represent. Theory 1 (1997), 83-114.

\bibitem[W17]{W17} H. Watanabe, Crystal basis theory for a quantum symmetric pair $(\U,\Uj)$, to appear in Int. Math. Res. Not., arXiv:1704.01277.

\bibitem[X94]{X94} N. Xi, Representations of Affine Hecke Algebras, Lecture Notes in Mathematics, 1587. Springer-Verlag, Berlin, 1994. viii+137 pp.
\end{thebibliography}
\end{document}